\documentclass[a4paper, 11pt, twoside, final]{scrartcl}

\usepackage{preamble}
\usepackage{command}
\usepackage{hyphenation}

\usepackage[draft, todonotes={textsize=scriptsize}]{changes}
\setuptodonotes{color=Crimson, backgroundcolor=Crimson!50,
  bordercolor=Crimson, tickmarkheight=10pt}
\definechangesauthor[name=Robert Beinert, color=cyan]{RB}

\AUTHOR{Robert \pers{Beinert}\Inst{1}}
\SHORTAUTHOR{R.~Beinert}
\TITLE{Approximation of Curve-based Sleeve Functions in High Dimensions}
\SHORTTITLE{Curve-based Sleeve Functions}

\INSTITUTE{\Inst{1}\parbox[t]{0.95\linewidth}{Institut für Mathematik\\
    Technische Universität Berlin\\
    Straße des 17. Juni 136\\
    10623 Berlin, Germany } }

\CORRESPONDENCE{R. \pers{Beinert}: \email{beinert@math.tu.berlin.de}}

\ABSTRACT{Sleeve functions are generalizations of the well-established
  ridge functions that play a major role in the theory of partial
  differential equation, medical imaging, statistics, and neural
  networks.  Where ridge functions are non-linear, univariate
  functions of the distance to hyperplanes, sleeve functions are based
  on the squared distance to lower-dimensional manifolds.  The present
  work is a first step to study general sleeve functions by starting
  with sleeve functions based on finite-length curves.  To capture
  these curve-based sleeve functions, we propose and study a two-step
  method, where first the outer univariate function---the profile---is
  recovered, and second the underlying curve is represented by a
  polygonal chain.  Introducing a concept of well-separation, we
  ensure that the proposed method always terminates and approximate
  the true sleeve function with a certain quality.  Investigating the
  local geometry, we study an inexact version of our method and show
  its success under certain conditions.}

\KEYWORDS{Sleeve functions, generalized ridge functions, adaptive
  learning, multivariate approximation.}

\AMSCLASS{41A15, 41A30, 41A63, 65D15.}

\DATE{14th September 2021}

\begin{document}
\thispagestyle{plain}
\TitleHeader

\section{Introduction}
\label{sec:intro}

The capturing or approximation of multivariate functions is nowadays
one of the key elements to tackle a great number of scientific and
real-world problems.  A complicated function is here usually replaced
by a sufficient simple function to find the numerical solution of the
problem or to speed up the required numerical computations.
Especially, if the domain becomes high dimensional, the approximation
should rely on relatively few given function values.  Unfortunately,
the approximation of highly multivariate functions is hampered by the
so-called curse of dimensionality \cite{HNW11,NW09,Bel61}.  One of the
most impressive results given by \pers{Novak} \& \pers{Woźniakowski}
\cite{NW09} is here the intractability of the uniform approximation of
even smooth functions.

In many applications, the considered functions however possess
specific low-di\-men\-sion\-al structures that allow to overcome the curse
of dimensionality.  One popular approach assumes that the function of
interest looks like a ridge.  Mathematically, a \emph{ridge function}
$f \colon \BR^d \to \BR$ possesses the form
\begin{equation*}
  f(x) \coloneqq g(Ax)
  \qquad\text{with}\qquad
  A \in \BR^{\ell \times d},
\end{equation*}
where $\ell \ll d$.  This functions are especially constant along the
kernel of $A$.  The function $g$ is usually called the \emph{ridge
  profile} whereas $A$ is the \emph{ridge matrix}.  In the extreme
case $\ell = 1$, the ridge function becomes
\begin{equation*}
  f(x) \coloneqq g( \iProdn{a}{x})
  \qquad\text{with}\qquad
  a \in \BR^d.
\end{equation*}
Although this (vector-based) ridge functions are constant along a
$(d-1)$-dimensional subspace perpendicular to $a$, there are numerous
applications showing their usefulness.  For instance, they appear as
plane waves in the theory of partial differential equations
\cite{Joh81}, play a major role in computed tomography \cite{LS75},
occur in statistical regression theory \cite{DJ89,FS81}, and provide
the basis to analyse neural networks
\cite{Pin99,Ism15,Can99,JS21,Pet99,XC11}.  Further, the approximation
of a multivariate function by a sum of vector-based ridge functions is
a topic on its own and has been studied in
\cite{AAI19,JS21,KKM10,Kro97,Moi99,Mai10,Pet99}.  The approximation
properties of matrix-based ridge functions have been considered in
\cite{LP93}.

Due to their usefulness and approximation properties, the question
arises how to learn the ridge profile and the ridge matrix/vector from
certain function evaluations.  A first step into this direction have
been done by \pers{DeVore}, \pers{Petrova} \& \pers{Wojtaszczyk}
\cite{DPW11}, who study the approximation of functions
$f \colon \BR^d \to \BR$ only depending on a small set of active
variables, \ie\ $f(x_1, \dots, x_d) = g(x_{i_1}, \dots, x_{i_\ell})$
with $\ell \ll d$,---ridge function, whose ridge matrix consists of
only unit vectors.  For this, \pers{DeVore}, \pers{Petrova} \&
\pers{Wojtaszczyk} establish an algorithm to find the active variables
and approximate the profile.

For vector-based ridge functions, a first recovery algorithm using
function queries has been proposed by \pers{Cohen} \etal\
\cite{CDDK+12}, where first the ridge profile is approximated, and
afterwards the ridge vector with non-negative entries is recovered
using compressed sensing techniques.  Overall, the algorithm performs
comparable to the approximation of univariate, continuous functions
with respect to the approximation rate.  Another approach to tackle
the recovery problem is to exploit the directional derivatives
\begin{equation*}
  \iProdn{\nabla f(x)}{\phi} = \dot g(\iProd{a}{x}) \, \iProd{a}{x}
\end{equation*}
or, more precisely, to approximate them by finite differences.  Using
them, \pers{Fornasier}, \pers{Schnass} \& \pers{Vybíral} \cite{FSV12}
study vector-based ridge functions on the ball.  Their results have then
be extended to the cube by \pers{Kolleck} \& \pers{Vybíral}
\cite{KV15}.  In order to apply the tools from compress sensing, the
previous works assume that the ridge vector is sparse or nearly
sparse.  The algorithm proposed by \pers{Tyagi} \& \pers{Cevher}
\cite{TV12} uses techniques form low-rank matrix recovery---more
precisely, the \pers{Dantzig} selector---to overcome the
compressibility assumption.  Moreover, the recovery of vector-based
ridge functions has been studied by \pers{Mayer}, \pers{Ullrich} \&
\pers{Vybíral} \cite{MUV15}, who show that the derivative of the ridge
profile has to be bounded from below away from zero and that this
condition is necessary in order to reduce the sampling complexity.

Essentially, a vector-based ridge function may be interpreted as a
function of the distance to a $(d-1)$-dimensional subspace, \ie\
\begin{equation*}
  f(x) = g (\iProdn{a}{x}) = g(\dist(x,\mathcal A))
  \qquad\text{with}\qquad
  \mathcal A \coloneqq \{x : x \perp a\};
\end{equation*}
so one possible generalization is to replace the subspace by another
set.  Since the level sets of the resulting $f$ are now inflated
versions of $\mathcal A$ and do not look like a ridge anymore, this
type of functions has been called \emph{sleeve functions} by
\pers{Keiper} \cite{Kei19}.  More precisely, \pers{Keiper} defined a
sleeve function on the basis of the squared distance, \ie\
\begin{equation*}
  f(x) \coloneqq g \bigl( \dist(x, \mathcal A)^2 \bigr),
\end{equation*}
where $\mathcal A$ is a low-dimensional manifold in $\BR^d$ like an
$\ell$-dimensional subspace $L$ with $\ell < d$.  For the later case,
the gradient of $f$ becomes perpendicular to the underlying subspace
$L$ in analogy of the gradient of a vector-based ridge function.  This
implies that the tangent plane $\nabla f(x)^\perp$ contains $L$.  On
the basis of this observation \pers{Keiper} derived an adaptive
algorithm to recover $L$ by approximating the gradient $\nabla f$ at
random points by finite differences.  An alternative approach to learn
$L$ is to solve an optimization problem over the \PGrassmann[ian]
\cite{Kei19}.

\paragraph{Our Contribution}

The focus of the present work is to study sleeve functions that are
based on non-linear underlying structures.  More precisely, we are
interested in learning sleeve functions based on \PJordan\ arcs, which
are injective mappings $\gamma : [0,\ell(\gamma)] \to \BR^d$, where
$\ell$ is the length of the arc.  \PJordan\ arcs are thus
non-self-intersecting, finite-length curves.  We require that our
underlying \PJordan\ arc is at least twice continuously
differentiable---henceforth called (\PJordan) $C^2$-arc or
$C^2$-curve.  Formally, the corresponding sleeve function is defined
as follows.

\begin{definition}[Sleeve Function]
  Let $g$ be in $C^2([0,\infty])$, and let $\gamma : [0, \ell(\gamma)] \to
  \BR^d$ be a \PJordan\ $C^2$-arc.  The function $f : \BR^d \to \BR$ given by
  \begin{equation*}
    f(x) \coloneqq g \bigl( (\dist(x,\gamma))^2 \bigr)
  \end{equation*}
  is called a \emph{curve-based sleeve function}.
\end{definition}

Our main interest is to capture the \emph{sleeve profile} $g$ and the
\emph{underlying curve} $\gamma$ numerically from point and gradient
queries, where the gradient may also be approximated by finite
differences.  To analyse the proposed algorithm, we require that the
profile is bounded from below away from zero at the origin and that
the \PJordan\ curve remains non-self-intersecting if it is inflated to
some extend.  Our central contributions are the following:

\begin{itemize}
\item We propose an adaptive, two-step learning algorithm to
  approximate the profile and to recover the underlying curve on the
  basis of projections that are computed form function and gradient
  queries.  The algorithm always terminate and captures the underlying
  curve and profile up to given approximation errors.
\item We analyse the effect of inaccurately computed projections to the
  proposed method recovering the underlying curve and show that the
  additional error can be controlled under suitable assumptions.
\item We derive a uniform bound of the approximation error for
  the composed two steps of the proposed method.
\end{itemize}

\paragraph{Roadmap} After starting with some preliminaries in
Section~\ref{sec:preliminaries}, we introduce the concept of
well-separated \PJordan\ curves enforcing an extended
non-self-intersecting property and an bounded curvature.  The
consequences regarding projection and distance to a curve are studied
in Section~\ref{sec:well-separ-curv}.  Using the well-separation, we
calculate the maximal approximation error caused by replacing an arc
by a polygonal chain, see Section~\ref{sec:curves-poly-chain}.  In
Section~\ref{sec:reconstr-underly-cur}, we derive the second step of
our capturing algorithm, which recovers the underlying curve by
computing projections from function and gradient queries, and study
the influence of numerical errors during the projection.  Capturing
the sleeve profile, which builds the first step of the method, we
propose an adaptive learning algorithm and bound the corresponding
approximation error in Section~\ref{sec:ident-curve-sleeve}.  We
conclude with some numerical experiments in
Section~\ref{sec:numerical-example} showing that the method can be
implemented and considering some special cases.

\section{Preliminaries}
\label{sec:preliminaries}

To simplify the notation, we throughout---up to the numeric
section---assume that the \PJordan\ arc $\gamma$ is parameterized via
the arclength, \ie\ $\gamma : [0, \ell(\gamma)] \to \BR^d$, where
$\ell(\gamma)$ denotes the length.  As a consequence the tangent
vector is normalized, \ie\ $\pNormn{\dot\gamma(t)} = 1$ for
$t \in [0, \ell(\gamma)]$, and the norm of the second derivative
coincides with the (unsigned) curvature
$\kappa_\gamma(t) = \pNormn{\ddot\gamma(t)}$ for
$t \in [0, \ell(\gamma)]$.  The \emph{end points} of $\gamma$ are just
$\gamma(0)$ and $\gamma(\ell(\gamma))$.  We refer to the remaining
points $\gamma(t)$ with $t \in (0, \ell(\gamma))$ as \emph{inner
  point} of the curve.

The distance of a point $x$ to $\gamma$ is
now defined by $\dist(x, \gamma) \coloneqq \dist(x, \ran \gamma)$,
where $\ran$ denotes the image of $\gamma$ in $\BR^d$, and where
$\dist(x, A) \coloneqq \inf_{y \in A} \pNormn{x - y}$ for general
$A \subset \BR^d$ and the \PEuclid[ean] norm.  Based on the distance,
the projection onto $\gamma$ is defined as the set-valued map
\begin{equation*}
  \proj_\gamma(x) \coloneqq
  \bigl\{ y : \pNormn{x - y} = \dist(x, \gamma) \bigr\}.
  \addmathskip
\end{equation*}
Because of the closedness of $\ran \gamma$, the distance is always
attained; so the projection is never empty.  If the projection is
single-valued, we may interpret $\proj_\gamma(x)$ as vector or point
unstated, \ie\ we may write $y = \proj_\gamma (x)$ with
\raisebox{0pt}[0pt][0pt]{$y \in \BR^d$}.  If the (single-valued)
projection can only be computed up to $\epsilon > 0$, we use the
notation $y \approx_\epsilon \proj_\gamma(x)$, where
$y \approx_\epsilon z$ means $\pNormn{x-z} \le \epsilon$.  Denoting
the cardinality of a set by $\#[\cdot]$, we call $x$ \emph{ambiguous}
or an \emph{ambiguity point} if $\# [\proj_\gamma(x)] > 1$ and
\emph{unambiguous} or an \emph{unambiguity point} otherwise.

To measure the distance between two curves, we rely on the \PHausdorff\
distance.  More generally, the \PHausdorff\ distance between two
arbitrary sets $A \subset \BR^d$ and $B \subset \BR^d$ is defined as
\begin{equation*}
  d_{\id H}(A,B) \coloneqq \max \bigl\{ \sup_{x \in A} \dist(x,B), \sup_{y
    \in B} \dist(y, A) \bigr\}.
\end{equation*}
For two curves $\gamma$ and $\tilde\gamma$, we define the distance
between them as the \PHausdorff\ distance of their images, \ie\
\begin{equation*}
  \dist(\gamma, \tilde\gamma)
  \coloneqq d_{\id H}(\ran \gamma, \ran \tilde \gamma).
\end{equation*}

For two points $P,Q \in \BR^d$, we denote by
$\overrightarrow{PQ} = (PQ)^{\rightarrow}$ the vector from $P$ to $Q$,
\ie\ $(PQ)^\rightarrow = Q - P$ where $P$ and $Q$ are interpreted as
vectors themselves.  Likewise, we define the distance and projection
of a point to a curve as above.  The \PEuclid[ean] distance between
$P$ and $Q$ is just given by
$\dist(P,Q) \coloneqq \pNormn{(PQ)^\rightarrow}$.  The linear line
segment between $P$, $Q \in \BR^d$ is denoted by $[PQ]$. For points on
a curve, \ie\ $P = \gamma(t)$ and $Q = \gamma(s)$, the arc of $\gamma$
between them is denoted by $\gamma_{[PQ]} \coloneqq \gamma([t,s])$.

Besides the \PEuclid[ean] norm $\pNormn{\cdot}$, we use the
\PChebyshev\ norm $\pNormn{\cdot}_\infty$ for vectors and the
\PFrobenius\ norm $\FNormn{\cdot}$ for matrices. The unit vectors are
denoted by $e_n$ with $n = 1, \dots, d$, and the all-ones vector by
$\Vek 1$.  We refer to the identity matrix as $I$.  For the ball with
radius $\epsilon$ around $x$, we write $B_\epsilon(x)$ and, for the
open ball, $\mathring B_\epsilon(x)$.  The non-negative and
non-positive real numbers are denoted by $\BR_+$ and $\BR_-$
respectively---both contain zero.  The cone of a set $A \subset \BR^d$
is defined as $\cone A \coloneqq \{tx : x \in A, t \ge 0\}$.  Finally,
we denote by $C^r(X,Y)$ the $r$-times continuously differentiable
mappings from $X$ to $Y$.  If $X$ and $Y$ are obvious, we may write
$C^r(X)$ or even $C^r$.

\section{Well-Separation of \PJordan\ Curves}
\label{sec:well-separ-curv}

Since our approximation of the unknown curve $\gamma$ will be based on
projections, we have to ensure that these are well defined in a
neighbourhood of $\gamma$.  For this, the \PJordan\ curve is not
allowed to intersect itself even if it is inflated to some extend.  To
express this assumption mathematical, we use tangential and normal
cones.  For the image of a $C^2$-curve, the \emph{tangential cone}
\cite[Def~6.1]{RW09} becomes
\begin{equation*}
  T_\gamma (t)
  \coloneqq
  \begin{cases}
    \cone \{\dot\gamma(0+)\} & \text{if} \; t = 0,\\
    \Span \{\dot\gamma(t)\} & \text{if} \; t \in (0, \ell(\gamma)),\\
    \cone \{-\dot\gamma(\ell_\gamma-)\} & \text{if} \, t = \ell(\gamma);
  \end{cases}
\end{equation*}
so $T_\gamma(t)$ is usually spanned by the tangent and is the ray of the
left-hand or right-hand derivative at the end points.  The
\emph{(regular) normal cone} \cite[Prop~6.5]{RW09} is now given by
\begin{equation*}
  N_\gamma (t) \coloneqq
  \bigl\{
  v : \iProd{v}{w} \le 0 \; \text{for all} \; w \in T_\gamma(t).
  \bigr\};
\end{equation*}
so $N_\gamma(t)$ is usually the hyperplane orthogonal to the tangent
$\dot\gamma(t)$ and a half-space at the end points.

\begin{definition}[Well-Separation]
  A \PJordan\ curve $\gamma \colon [0, \ell(\gamma)] \to \BR^d$ is
  $\rho$-separated if
  \begin{equation*}
    \mathring B_\rho (\gamma(t) + \rho v) \cap \ran \gamma = \emptyset
  \end{equation*}
  for all $t \in [0, \ell(\gamma)]$ and $v \in N_\gamma(t)$ with
  $\pNormn{v} = 1$.
\end{definition}

Figuratively, the well-separation means that $\gamma$ is not allowed
to intersect with an inflated circle around the curve.  Henceforth, we
call the inflated circle
\begin{equation*}
  R_\gamma (t) \coloneqq
  \bigcup_{\substack{v \perp \dot\gamma(t)\\ \pNormn{v} = 1}}
  \mathring B_\rho (\gamma(t) + \rho v)
\end{equation*}
the \emph{normal ring} of $\gamma$ at $t\in (0, \ell(\gamma))$.  In
three dimensions, the normal ring looks like a threaded horn torus or
water wing.  At the end points, the normal ring is closed with a
half-ball.  Up to a local neighbourhood, the points of $\gamma$ are
thus well-separated by a distance of $2\rho$ at the least.  Further,
the curvature of $\gamma$ is bounded by $\nicefrac1\rho$ since the
osculating circle would then be included in the boundary of the normal
ring.

The well-separation ensures the single-valueness of the projection
within a certain neighbourhood of the curve.

\begin{theorem}[Single-Valued Projection]
  \label{thm:single-proj}
  Let the \PJordan\ $C^2$-arc $\gamma : [0, \ell(\gamma)] \to \BR^d$ be
  $\rho$-separated, and let $x \in \BR^d$ be a point with
  $\dist(x, \gamma) < \rho$.  The projection $\proj_\gamma(x)$ is then
  single-valued and thus unique in an open neighbourhood of $x$.
\end{theorem}

\begin{proof}
  If $P = \gamma(t)$ is contained in $\proj_\gamma(x)$, then $x-P$ has
  to be in the normal cone $N_\gamma(t)$, see \cite[Ex~6.16]{RW09}.
  Therefore, the ball $B_\sigma \coloneqq B_\sigma (x)$ with radius
  $\sigma \coloneqq \pNormn{x-P} < \rho$ is contained in
  $B_\rho \coloneqq B_\rho (P + \nicefrac{(x - P)}{\pNormn{x - P}})$,
  which is the closure of a ball in the normal ring.  Since the
  interior of $B_\rho$ contains no points of $\gamma$, and since
  $B_\sigma$ touches $\uppartial B_\rho$ at exactly one point, the
  intersection $B_\sigma \cap \ran \gamma$ consists only of the point
  $P$ showing that the projection is single-valued.  The same
  argumentation holds for all points in $\mathring B_\epsilon(x)$ with
  $\epsilon < \rho - \dist(x,\gamma)$.  \qed
\end{proof}

More general, if $x \in \BR^d$ is no ambiguity point and does not lie
on $\gamma$, we always find a small neighbourhood where the projection
is single-valued too.  Moreover, the set of ambiguity points is a set
of measure zero.  To study the set of ambiguity points, we use that
the projection and distance onto and to a \PJordan\ $C^2$-curve is
differentiable for most unambiguity points.  Both results can be found
in \cite{DH94}, and we state them for our specific setting with
$C^2$-curves.

\begin{theorem}[{\pers{Dudek--Holly} \cite[Thm~4.1]{DH94}}]
  \label{thm:diff-proj}
  Let $\gamma$ be a \PJordan\ $C^2$-arc, and let $x \in \BR^d$ be a point
  within an open neighbourhood where the projection is single-valued.
  If $\proj_\gamma(x)$ in an inner point, then $\proj_\gamma$ is
  differentiable at $x$.
\end{theorem}

The restriction to a point that is projected to an inner point is here
crucial since the projection becomes undifferentiable at the end points. 

\begin{counterexample}[End Points]
  Consider the curve or line segment $\gamma(t) \coloneqq (t,0)^\T$
  with $t \in [0,1]$.  For $x \coloneqq (0,1)^\T$, the derivative in
  direction $(1,0)^\T$ is $(1,0)^\T$ but $(0,0)^\T$ in direction
  $(-1,0)^\T$.  Thus the projection is not differentiable at points
  $x \coloneqq \gamma(0) + v$ with $v \perp \dot\gamma(0+)$ and
  $\proj_\gamma(x) = \gamma(0)$.  \qed
\end{counterexample}

The distance to a curve is obviously not differentiable for points on
the curve.  In difference to the projection, the distance is
differentiable for points projected onto the end points.

\begin{theorem}[{\pers{Dudek--Holly}, \cite[Prop~4.4]{DH94}}]
  \label{thm:diff-dist}
  Let $\gamma$ be a \PJordan\ $C^2$-arc, and let $x \in \BR^d$ be a point
  within an open neighbourhood where the projection is single-valued.
  If $x \not\in \ran \gamma$, then $\dist_\gamma$ is differentiable at
  $x$ with
  \begin{equation*}
    \nabla_x \dist(x, \gamma) =
    \frac{x - \proj_\gamma(x)}{\pNormn{x - \proj_\gamma(x)}}.
  \end{equation*}
\end{theorem}

\begin{proof}
  For points $x$ projected to inner points, then statement has been
  established in \cite{DH94}.  If $x$ is in a neighbourhood projected
  to one end point, the statement is clear since the distance becomes
  the \PEuclid{ean} distance to the end point.  The interesting points
  are thus $x \coloneqq \gamma(0) + v$ with $v \perp \dot\gamma(0+)$ and
  $\proj_\gamma(x) = \gamma(0)$.  On one side of the affine hyperplane
  corresponding to $\dot\gamma(0)$ the points are projected to an inner
  point and on the other side to the end point.  The corresponding
  derivatives are
  \begin{equation*}
    \nabla_y \dist(y,\gamma) = \frac{y - \proj_\gamma(y)}{\pNormn{y -
        \proj_\gamma(y)}}
    \qquad\text{and}\qquad
    \nabla_y \dist(y,\gamma) = \frac{y - \gamma(0)}{\pNormn{y -
        \gamma(0)}}.
  \end{equation*}
  For $y \to x$ in both half-spaces, the derivatives converges to the
  same value since the projection is continuous.  \qed
\end{proof}

The ambiguity points with respect to a \PJordan\  $C^2$-curve have
a benign structure.  The restriction
$A_2 \coloneqq \{ x \in \BR^d : \#[ \proj_\gamma(x)] = 2 \}$ of the
ambiguity set
$A \coloneqq \{ x \in \BR^d : \#[ \proj_\gamma(x)] > 1 \}$ consisting
of all the points with exactly two projections has \PLebesgue\ measure
zero. 

\begin{lemma}[Ambiguity Points]
  \label{lem:2-amb-pts}
  Let $\gamma$ be a finite-length \PJordan\ $C^2$-arc.  Then the subset
  $A_2 \coloneqq \{ x \in \BR^d : \#[ \proj_\gamma(x)] = 2 \}$ has
  \PLebesgue\ measure zero.
\end{lemma}

\begin{proof}
  Let $x$ be an ambiguity point in $A_2$ with projection $P_1$ and
  $P_2$.  Since the distance function is continuous, we find a small
  open neighbourhood $U_x$ such that $\dist(y,\gamma)$ is attained by
  a curve point near $P_1$ and/or $P_2$, \ie\
  $\dist(y,\gamma) = \min\{ \dist(y,\gamma_1), \dist(y,\gamma_2) \}$,
  where $\gamma_1$ and $\gamma_2$ are small arcs around $P_1$ and
  $P_2$.  Further $U_x$ may be chosen small enough such that the
  projection to a single arc $\gamma_1$ or $\gamma_2$ is single-valued
  in $U_x$ such that $\proj_{\gamma_1}$ and $\proj_{\gamma_2}$ become
  continuously differentiable by \thref{thm:diff-proj}.  The ambiguity
  points in $U_x$ are the zeros of the function
  \begin{equation*}
    F \colon U_x \to \BR :
    y \mapsto \dist(y,\gamma_1) - \dist(y, \gamma_2).
  \end{equation*}
  Because the gradient
  \begin{equation*}
    \nabla F(x) = \frac{x - P_1}{\pNormn{x - P_1}}
    - \frac{x - P_2}{\pNormn{x - P_2}}
    \addmathskip
  \end{equation*}
  is non-zero, $P_1$ and $P_2$ are distinct points, \PDini's implicit
  function theorem \cite[Thm 1B.1]{DR09} states that the ambiguity set
  $A_2$ in an open neighbourhood $\tilde U_x \subset U_x$ around $x$
  is the realization of a continuously differentiable map
  $a_x \colon \BR^{d-1} \to \BR^d$ and thus a \PLebesgue\ zero set by
  \PSard' theorem \cite{Sar42}.  Since the \PEuclid{ean} $\BR^d$ is
  second-countable, already countably many set $U_{x_n}$ cover $A_2$,
  whose union is again a \PLebesgue\ zero set.  \qed
\end{proof}

\begin{proposition}[Ambiguity Points]
  \label{prop:amb-pts}
  Let $\gamma$ be a \PJordan\ $C^2$-arc.  Then the ambiguity set
  $A \coloneqq \{ x \in \BR^d : \#[ \proj_\gamma(x)] > 1 \}$ is the
  closure of $A_2$.
\end{proposition}

\begin{proof}
  Since the distance to the curve is continuous, the points in
  $\overline A_2$ are ambiguous.  To show $A \subset \overline A_2$,
  we take an ambiguity point $x$ with $\#[\proj_\gamma(x)] > 2$.  In
  two dimensions, the set $\proj_\gamma(x)$ is located on a circle.
  Since $\gamma$ is not closed, we can either shrink the circle and
  move it into a gap between to projection points or, if
  $\proj_\gamma(x)$ lie on a half-sphere, we can move the circle
  outwards and enlarge it, see Figure~\ref{fig:amb}.  In both cases,
  the center $y$ of the deformed circle is contained in $A_2$.  By
  controlling the radius, the center may be arbitrary close to $x$.
  This construction generalizes to $\BR^d$ by changing the radius and
  moving the sphere containing the projection points in several steps.
  \qed
\end{proof}

\begin{figure}[t]
  \centering
  \subfloat[Shrink and move the circle into a gap.]
  {\includegraphics{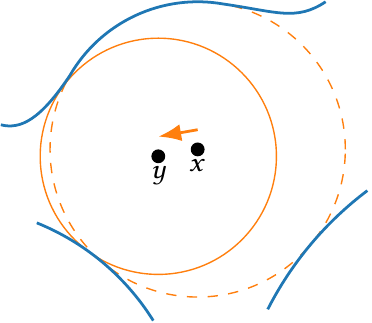}}
  \qquad
  \subfloat[Enlarge and move the circle outwards.]
  {\includegraphics{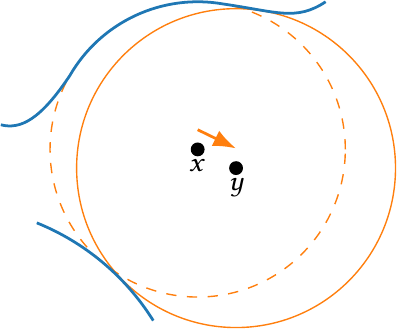}}
  \caption{The projections of the ambiguity point $x$ onto the
    {\color{Pri}curve $\gamma$} lie on a {\color{Sec}circle}.
    Changing the radius and moving the circle around, we find an
    arbitrary close point with exactly two projection.}
  \label{fig:amb}
\end{figure}

Figuratively, higher ambiguities with $\#[\proj_\gamma(x)] > 2$ occur
at points, where the charts constructed by the implicit function
theorem are glued together.  Since $A_2$ is locally a hypersurface,
the \PLebesgue\ measure of the closure remains
zero.  On the basis of a different argumentation, the statement has
already by shown by \pers{Hastie} \& \pers{Stuetzle} \cite{Has84,HS89}
for smooth curves.

\begin{theorem}[\pers{Hastie}--\pers{Stuetzle}, {\cite[Prop~6]{HS89}}]
  \label{thm:amb-pts}
  Let $\gamma$ be a \PJordan\ $C^2$-arc.  Then the ambiguity set
  $A \coloneqq \{ x \in \BR^d : \#[ \proj_\gamma(x)] > 1 \}$ has
  \PLebesgue\ measure zero.
\end{theorem}

\section{Approximation by Polygonal Chains}
\label{sec:curves-poly-chain}

If the curvature of a planar curve $\gamma$ is
bounded by $\kappa_{\max}$, then the approximation error between an
arc $\gamma_{[PQ]}$ and the line segment $[PQ]$ is bounded by
\begin{equation*}
  \dist(\gamma_{[PQ]}, [PQ]) \le \frac{\ell^2(\gamma_{[PQ]})}{2} \,
  \kappa_{\max}, 
\end{equation*}
see \cite{BDDD04}.  Since we however have no information about the
current arc $\gamma_{[PQ]}$ living in $\BR^d$, we use a more geometric
consideration to estimate the maximal approximation error.

\begin{theorem}[Approximation Error]
  \label{thm:app-err}
  Let $\gamma \colon [0, \ell(\gamma)] \to \BR^d$ be a
  $\rho$-separated \PJordan\ $C^2$-arc, and let $P$ and $Q$ be
  distinct points on $\gamma$ such that
  $h \coloneqq \dist(P,Q) < 2 \rho$.  The maximal approximation error
  is then bounded by
  \begin{equation*}
    \dist(\gamma_{[PQ]}, [PQ]) \le \rho - \sqrt{\rho^2 -
      \tfrac{h^2}{4}}. 
  \end{equation*}
\end{theorem}

\begin{figure}[t]
  \centering
  \includegraphics{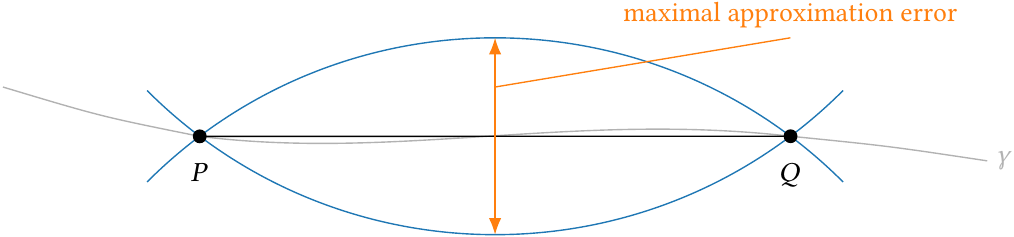}
  \caption{The two circles correspond to the {\color{Pri}worst-case
    normal rings} at $P$ that touch $Q$ and vice versa.  The
    {\color{Sec}maximal approximation error} may be attained in the
    middle of $[PQ]$.}
  \label{fig:appro-planar}
\end{figure}

\begin{proof}
  We first discuss the planar setting, where the normal ring becomes
  the union of two open discs of radius $\rho$ for inner points and is
  extended by an open half-disc of radius $2\rho$ at the end points.
  By the $\rho$-separation, the normal rings of $\gamma_{[PQ]}$ are
  not allowed to cover $P$ or $Q$.  To estimate the approximation
  error, let us consider the worst-case normal rings at $P$ that also
  touch $Q$ and vice versa, \ie\ the two discs of radius $\rho$
  touching $P$ and $Q$.  The situation is schematically shown in
  Figure~\ref{fig:appro-planar}.  The main observation is now that the
  arc $\gamma_{[PQ]}$ has to lie in the intersection of the two discs.
  For this, we notice two things:
  \begin{enumerate*}
  \item $\gamma$ cannot cross the boundary of the intersection except
    at $P$ and $Q$ since otherwise the normal ring at the crossing
    point would cover $P$ or $Q$;
  \item $\gamma$ cannot start in the intersection, leave at $P$, go
    around the intersection, and enter again at $Q$ due to the end
    point condition of the $\rho$-separability.
  \end{enumerate*}
  Since the two circular segment shown in
  Figure~\ref{fig:appro-planar} are possible paths for
  $\gamma_{[PQ]}$, the maximal \PHausdorff\ approximation error is
  attained in the middle of $[PQ]$.  The \PPythagoras[ean] theorem now
  yields the assertion.

  Similarly in higher dimensions, $\gamma_{[PQ]}$ is contained in the
  intersections of all open balls of radius $\rho$ touching $P$ and
  $Q$.  Possible arcs form $P$ to $Q$ with maximal \PHausdorff\
  approximation error are the circular segments of a two-dimensional
  cut through the intersection containing $P$ and $Q$, which looks
  exactly as in the planar setting.  The assertion again follows by
  the \PPythagoras[ean] theorem.  \qed
\end{proof}

If the end points $P$ and $Q$ of $\gamma_{[PQ]}$ are inaccurate,
the error analysis may be adapted by enlarging the area that contains
the true arc.

\begin{corollary}[Approximation Error]
  \label{cor:app-err}
  Let $\gamma \colon [0, \ell(\gamma)] \to \BR^d$ be an
  $\rho$-separated \PJordan\ $C^2$-arc, and let $\tilde P$, $\tilde Q$ with
  $\dist(\tilde P , P) \le \epsilon$,
  $\dist(\tilde Q , Q) \le \epsilon$, and
  $h \coloneqq \dist(\tilde P, \tilde Q) < 2 (\rho - \epsilon)$ be
  approximations of the curve points $P$, $Q$. The maximal
  approximation error of the arc $\gamma_{[PQ]}$ is bounded by
  \begin{equation*}
    \dist(\gamma_{[PQ]}, [\tilde P \tilde Q]) \le \rho + \epsilon -
    \sqrt{\rho^2 - 
      \tfrac{(h+2\epsilon)^2}{4}}.
  \end{equation*}
\end{corollary}

\begin{proof}
  Extend the line $[\tilde P \tilde Q]$ by a segment of length
  $\epsilon$ on both ends and move the constructed circular segments
  in every two-dimensional cut away from the approximation line by
  $\epsilon$ to cover the true end points of the arc $\gamma_{[PQ]}$.
  \qed
\end{proof}

The above estimate only considers the distance between the curve points
and the maximal curvature encoded in the well-separation.  It is
possible to improve the bounds using tangents at $P$ and $Q$
additionally.  In doing so, the main improvement can been seen for the
approximation of long arcs $\gamma_{[PQ]}$.  To control the \PHausdorff\
approximation error numerically without knowing the curve itself, the
length of the line segments has to be rather small such that the angle
between the tangents and the direction of the polygonal line segment
is negligible.

\section{Reconstruction of the Underlying \PJordan\ Curve}
\label{sec:reconstr-underly-cur}

The identification of a curve-based sleeve function consists of two
central part.  One the one side, we have the approximate the unknown
structure function $g$ and, on the other side, the underlying curve
$\gamma$.  Assume for this section that the differentiable structure
function $g \colon [0,\infty] \to \BR$ is strictly monotonically
increasing with $g(0) = 0$ and $\dot g \gg 0$ and is known in advance.
For any point $x \in \ran \gamma$, the sleeve function $f$ thus
becomes zero.  Otherwise, if $x$ is unambiguous, \thref{thm:diff-dist}
ensures that $f$ is differentiable with gradient 
\begin{align*}
  \nabla f(x)
  &= 2 \dot g\bigl( (\dist(x, \gamma))^2 \bigr) \,
    \dist(x, \gamma) \,
    \frac{x - \proj_\gamma(x)}{\pNormn{x - \proj_\gamma(x)}}
  \\[\fskip]
  &= 2 \dot g\bigl( (\dist(x, \gamma))^2 \bigr) \,
    \bigl( x - \proj_\gamma(x) \bigr);
\end{align*}
so the negative gradient points directly to $\proj_\gamma(x)$.
Moreover, the distance to the curve is encoded in the function value
$f(x)$ and may be determined by inverting the strictly monotone $g$.
Together, this allow us to compute the projection to the unknown curve
by evaluating $f$ and $\nabla f$.

\begin{algorithm}[Projection to Underlying \PJordan\ Curve]
  \label{alg:proj-cur}
  {\scshape Input:} $x \in \BR^d$, $f \colon \BR^d \to \BR$,
  $g \colon \BR^d \to \BR$.
  \begin{enumerate}[(1),nosep]
  \item Evaluate $z \coloneqq f(x)$.
  \item Compute $t \coloneqq g^{-1}(z)$.
  \item Determine $v \coloneqq \nabla f(x)$.
  \item Set $y \coloneqq x - \sqrt t \, \pNormn{v}^{-1} \, v$.
  \end{enumerate}
  {\scshape Output:} $y = \proj_\gamma(x)$.
\end{algorithm}

On the basis of this projection, we approximate the underlying curve
by a polygonal chain.  If $[PQ]$ is the last line segment, we consider
 $Q + s(PQ)^\rightarrow$, which may be seen as extension of
the segment in direction of $Q$, and project this point back to $\gamma$ by
\thref{alg:proj-cur}.  If the step size $s$ is chosen appropriately,
the distance of the new point to $Q$ can be controlled from below and
from above.

\begin{theorem}[Step Size Guarantee]
  \label{thm:step-size-guar}
  Let $\gamma \colon [0, \ell(\gamma)] \to \BR^d$ be a
  $\rho$-separated \PJordan\ $C^2$-arc, and let $P$ and $Q$ be distinct points
  on $\gamma$ such that $h \coloneqq \dist(P,Q) \le \eta$ for
  $\eta < \rho$.  Then $R \coloneqq \proj_\gamma(Q + s_* v)$ with
  $v \coloneqq (PQ)^\rightarrow / \dist(P,Q)$ and
  $s_* \coloneqq (\eta^2 + 2 \eta \rho) / (2 \rho + 2 \eta +h)$
  satisfies
  \begin{equation*}
    \tfrac 6{80} \, \eta \le \dist(R,Q) \le \eta,
    \addmathskip
  \end{equation*}
  or, if the lower bound does not hold, $R$ is an end point of
  $\gamma$.
\end{theorem}

\begin{figure}[t]
  \centering
  \includegraphics{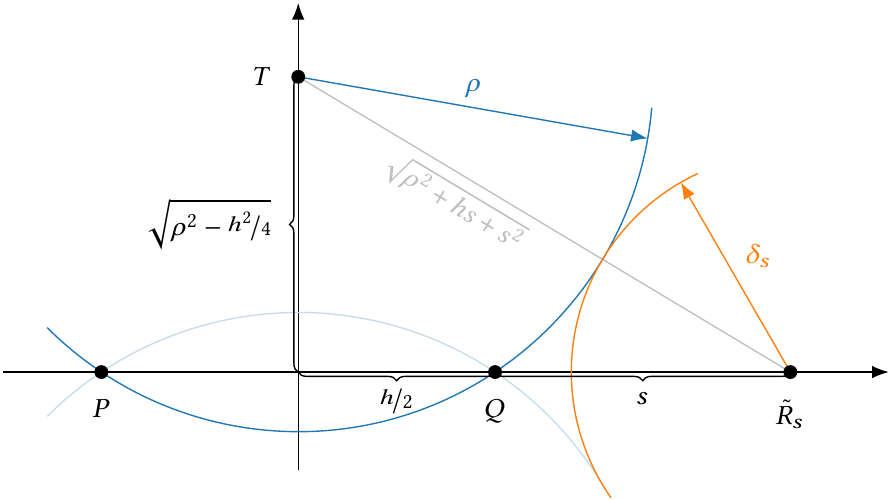}
  \caption{Calculation of the {\color{Sec}maximal projection distance}
    of the point $\tilde R_s$ onto $\gamma$ based on the
    {\color{Pri}worst-case normal rings}.}
  \label{fig:min-step}
\end{figure}
\begin{proof}
  Due to the restriction $\eta < \rho$, the parameter $s_*$ is bounded
  by $\nicefrac34 \, \rho$; so the distance between the point
  $\tilde R \coloneqq Q + s_* v$ and $\gamma$ is $\rho$ at the most,
  which ensures that the projection onto $\gamma$ is unique, and that
  $R$ is hence well defined.  Next, we study the maximal possible
  projection distance $\delta_s \coloneqq \dist(R_s, \tilde R_s)$ with
  $R_s \coloneqq \proj_\gamma(\tilde R_s)$ and
  $\tilde R_s \coloneqq Q + s v$ for arbitrary step sizes $s$.
  Because of the $\rho$-separation, the normal ring around $Q$ is not
  allowed to cover $P$; so the path of $\gamma$ after $Q$ is limited
  by considering all possible normal rings.  In two dimensions, the
  path is bounded by two discs, see Figure~\ref{fig:min-step}, in
  higher dimensions, by the union of all balls touching $P$ and $Q$.
  If $\gamma$ does not end near $Q$, the maximal possible projection
  distance is reached if $B_{\delta_s}(\tilde R_s)$ touches the union.
  The other way round, if $\delta_s$ becomes greater, then $\gamma$
  cannot leave the neighbourhood of $Q$ and $R_s$ is an end point.

  If $\gamma$ does not end, the maximal possible projection distance
  can be explicitely computed by considering a two-dimensional cut
  through the rotational invariant geometry as shown in
  Figure~\ref{fig:min-step} and is given due to the \PPythagoras[ean]
  theorem by
  \begin{equation*}
    \delta_s = \sqrt{\rho^2 + hs + s^2} - \rho.
  \end{equation*}
  For a fixed $s$, the
  distance between $R_s$ and $Q$ is thus bounded by
  \begin{equation*}
    s - \delta_s \le  \dist(R_s,Q) \le s + \delta_s.
  \end{equation*}
  The step size $s_*$ in the assertion is the unique solution of
  \begin{equation*}
    \sqrt{\rho^2 + hs + s^2} = \eta + \rho - s,
  \end{equation*}
  which ensures that the upper bound becomes exactly $\eta$.

  Inserting $s_*$ into the lower bound, we obtain
  \begin{equation*}
    s_* - \delta_{s_*}
    =
    \frac{\eta^2 + 2\eta\rho}{2\rho + 2\eta + h}
    - \sqrt{
      \rho^2 + h \, \frac{\eta^2 + 2 \eta \rho}{2 \rho + 2 \eta + h}
      + \Bigl( \frac{\eta^2 +2 \eta \rho}{2\rho + 2\eta + h} \Bigr)^2
    }
    + \rho.
  \end{equation*}
  Note that the lower bound is monotonically decreasing with respect
  to $h$ because the path of $\gamma$ is less restricted for larger
  $h$ resulting in an increasing maximal projection distance
  $\delta_s$; so the lower bound is limited by 
  \begin{equation*}
    s_* - \delta_{s_*}
    \ge
    \frac{\eta^2 + 2\eta\rho}{2\rho + 3\eta}
    - \sqrt{
      \rho^2 + \frac{\eta^3 + 2 \eta^2 \rho}{2 \rho + 3 \eta}
      + \Bigl( \frac{\eta^2 +2 \eta \rho}{2\rho + 3\eta} \Bigr)^2
    }
    + \rho > 0.
  \end{equation*}
  Rearranging the argument of the square root, we have
  \begin{equation*}
    t_1 \coloneqq \rho^2 + \frac{\eta^3 + 2 \eta^2 \rho}{2 \rho + 3 \eta}
    + \Bigl( \frac{\eta^2 +2 \eta \rho}{2\rho + 3\eta} \Bigr)^2
    =
    \frac{4 \rho^4 + 17 \eta^2 \rho^2 + 12 \eta \rho^3 + 12 \eta^3
      \rho + 4 \eta^4}{(2 \rho + 3 \eta)^2}.
  \end{equation*}
  Considering the square of the remaining term yield
  \begin{equation*}
    t_2
    \coloneqq
      \Bigl( \frac{\eta^2 + 2 \eta \rho}{2 \rho + 3 \eta} + \rho
      \Bigr)^2
      = \frac{(\eta^2 + 5 \eta \rho + 2 \rho^2)^2}{(2\rho + 3 \eta)^2}
    = \frac{4 \rho^4 + 29 \eta^2 \rho^2 + 20 \eta \rho^3 + 10 \eta^3
      \rho + \eta^4}{(2\rho + 3 \eta)^2}.
  \end{equation*}
  Applying the mean value theorem to the square root, and exploiting
  $t_2 > t_1$ and $\rho > \eta$, we finally obtain
  \begin{align*}
    s_* -  \delta_{s_*}
    &=\sqrt{t_2} - \sqrt{t_1}
      \ge \frac{t_2 - t_1}{2 \sqrt{t_2}}
    \\[\fskip]
    &\ge \frac{12 \eta^2 \rho^2 - 3 \eta^4 + 8 \eta \rho^3 - 2\eta^3
      \rho}{(2\rho + 3 \eta)^2} \,
      \frac{2 \rho + 3 \eta}{2(\eta^2 + 5 \eta \rho + 2 \rho^2)}
    \\[\fskip]
    &\ge \frac{9 \eta \rho^2 + 6 \rho^3}{2 \cdot (5 \rho) \cdot (8
      \rho^2)} \, \eta
      \ge \frac{6}{80} \, \eta.
      \tag*{\qed}
  \end{align*}
\end{proof}

The calculated optimal step size $s_*$ depends on
$h \coloneqq \dist(P,Q)$.  Figuratively, if the last step has been
small, the region where $\gamma$ runs and the related uncertainty
becomes smaller resulting in a greater step size and vice versa.
The guaranteed minimal and maximal step size allow us to move along
the unknown curve without getting stuck.  If the underlying curve is
well separated, we can, moreover, control the approximation error for
the obtained polygonal chain.

\begin{algorithm}[Underlying \PJordan\ Curve Approximation]
  \label{alg:curve-app}
  {\scshape Input:} $\rho > 0$, $E \in (0,\rho)$, $x_0 \in \BR^d$ with
  $\nabla f(x_0)$ is well defined.
  \\
  {\scshape Initialization:}
  \begin{enumerate}[(1),nosep]
  \item Set $\eta \coloneqq \min\{ \rho, 2 \sqrt{ \rho^2 - (\rho -
      E)^2} \}$.
  \item Compute $P_0 \coloneqq \proj_\gamma(x_0)$. 
  \item For $n = 1,2, \dots$ until $P_0 \ne P_1$ do:
    \begin{enumerate}[(a),nosep]
    \item Generate $v_n \in \BR^d$ with $\pNormn{v_n} = \nicefrac \eta
      2$ and $v_n \perp \Span\{ v_1, \dots, v_{n-1} \}$.
    \item Compute $P_1 \coloneqq \proj_\gamma (P_0 + v_n)$.  If
      $P_1 = P_0$, try $P_1 \coloneqq \proj_\gamma (P_0 - v_n)$.
    \end{enumerate}
  \end{enumerate}
  {\scshape Iterations:}
  \begin{enumerate}[resume*]
  \item For $n = 1,2,\dots$ until $P_n = P_{n-1}$ do:
    \begin{enumerate}[(a),nosep]
    \item Set $s \coloneqq (\eta^2 + 2 \eta \rho) / (2\rho +
        2\eta + h)$ with $h \coloneqq \dist(P_n, P_{n-1})$.
    \item Compute $P_{n+1} \coloneqq \proj_\gamma(P_n + s 
      v)$ with $v \coloneqq \overrightarrow{P_{n-1}P_n} /
      \dist(P_{n-1}, P_n)$.
    \end{enumerate}
  \item For $m = 0,1,\dots$ until $P_{-m} = P_{-m+1}$ do:
    \begin{enumerate}[(a),nosep]
    \item Set $s \coloneqq (\eta^2 + 2 \eta \rho) / (2\rho +
        2\eta + h)$ with $h \coloneqq \dist(P_{-m}, P_{-m+1})$.
    \item Compute $P_{-m-1} \coloneqq \proj_\gamma(P_{-m} + s v)$ with
      $v \coloneqq \overrightarrow{P_{-m+1}P_{-m}}) / \dist(P_{-m+1},
      P_{-m})$.
    \end{enumerate}
  \end{enumerate}
  {\scshape Output:} Polygonal chain $\tilde \gamma \coloneqq
  \bigcup_{k=-m+1}^{n-2} [P_{k} P_{k+1}]$ with $\dist(\tilde
  \gamma, \gamma) \le E$.
\end{algorithm}

This algorithm starts somewhere on the underlying curve and
iteratively moves in both direction until the end point of the curve
is reached.  Due to the guaranteed length of the appended line
segments, the procedure ends after finitely many steps.

\begin{theorem}[Termination]
  \label{thm:termi}
  Let $\gamma \colon [0, \ell(\gamma)] \to \BR^d$ be a
  $\rho$-separated \PJordan\ $C^2$-arc.  For any $E \in (0,\rho)$,
  \thref{alg:curve-app} yields a polygonal chain $\tilde \gamma$ with
  $\dist(\tilde \gamma, \gamma) \le E$ in finite time, \ie\
  \thref{alg:curve-app} terminates.
\end{theorem}

\begin{proof}
  The chosen step size
  $\eta \coloneqq \min\{ \rho, 2 (\rho^2 - (\rho - E)^2)^{\nicefrac12}
  \}$ ensures that all projections during the algorithm are well
  defined and single-valued by \thref{thm:single-proj}; so
  \thref{alg:curve-app} can be executed.  The maximal step size
  guarantee in \thref{thm:step-size-guar} further yields
  $\dist(P_k, P_{k+1}) \le \eta$ resulting in
  $\dist(\gamma_{[P_k P_{k+1}]}, [P_k P_{k+1}]) \le E$ for all
  line segments by \thref{thm:app-err}.  Since the length of the line
  segments $[P_k P_{k+1}]$ is bounded from below by
  \thref{thm:step-size-guar}, and since $\gamma$ has a finite length,
  the iterations (4) and (5) terminate as soon as the end point is
  reached, which happens in finitely many steps.  \qed
\end{proof}

If the required projections become inexact, in the worst case, we are
maybe not able to recover the underlying \PJordan\ curve.  For this reason, we
now study the caused errors and instabilities in more detail.  The
central idea is to adapt the calculations in the proof of
\thref{thm:step-size-guar}, where the maximal possible projection
distance of the point $\tilde R_s \coloneqq Q + s v$ with
$v \coloneqq (PQ)^\rightarrow / \dist(P,Q)$ has been computed.  If the
projections are inexact, then we can merely determine the points $P$
and $Q$ in Figure~\ref{fig:min-step} up to a small neighbourhood---say
up to an $\epsilon$-ball; so instead of using $\rho$-balls touching
$P$ and $Q$ to guarantee the minimal and maximal step size, we
consider $\rho$-balls intersecting $B_\epsilon(P)$ and
$B_\epsilon(Q)$.  To simplify the calculations, we restrict ourselves
to a specific set of balls that results form the following
construction, see Figure~\ref{fig:min-step:rot}:

\begin{figure}[t]
  \centering
  \includegraphics{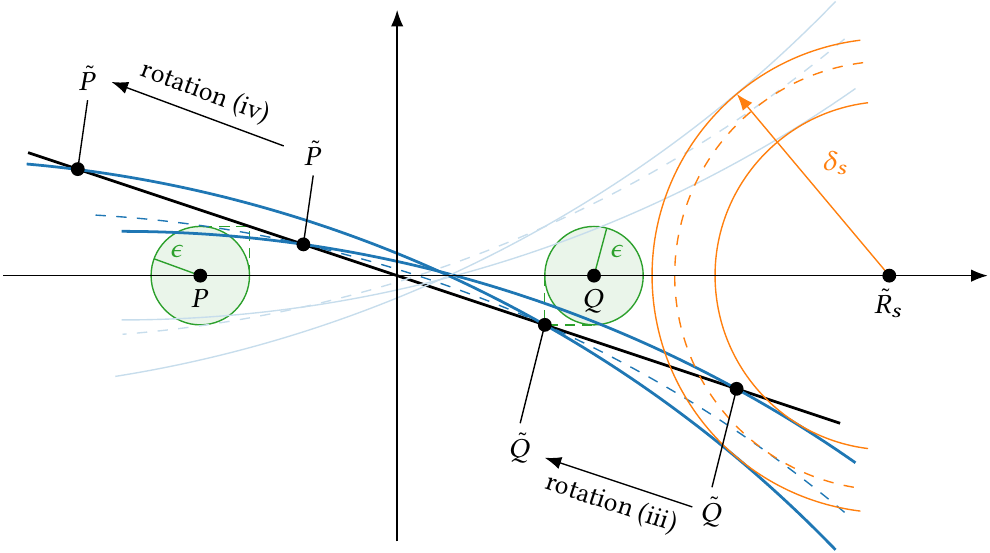}
  \caption{Rotating a {\color{Pri}specific ball} of the normal ring
    intersecting the {\color{Ter}$\epsilon$-balls} around $P$ and $Q$
    to standardize the local scenery.  Both rotations (iii) and (iv)
    enlarge the estimation of the {\color{Sec}maximal projection
      distance} for $\tilde R_s$.}
  \label{fig:min-step:rot}
\end{figure}

\begin{enumerate}[(i)]
\item take some $\rho$-ball intersecting $B_\epsilon(P)$ and
  $B_\epsilon(Q)$, and consider the plane through $P$, $Q$, and the
  centre of the ball;
\item denoting the intersection points of the $\rho$-ball with the line
  through $P + \epsilon \Vek 1$ and $Q - \epsilon \Vek 1$ by $\tilde
  P$ and $\tilde Q$;
\item rotate the ball around $\tilde P$ such that intersection $\tilde
  Q$ becomes $Q - \epsilon \Vek 1$, which enlarges the maximal
  projection distance; and
\item\label{it:fin-rot} rotate the ball around the new $\tilde Q$ and
  move $\tilde P$ such that $\dist(\tilde P, \tilde Q)$ becomes
  $h + 2 \epsilon$ with $h \coloneqq \dist(P,Q)$.
\end{enumerate}

Without loss of generality, we have assumed that the centre is located
below $P$ and $Q$ within the two-dimensional cut.  Analogous rotations
can be applied if the centre lies above $P$ and $Q$.  The rotation
\ref{it:fin-rot} only enlarges the maximal projection distance
corresponding to the original $\rho$-ball if $\rho$ is large compared
to $h$.

\begin{figure}[tp]
  \centering
  \includegraphics{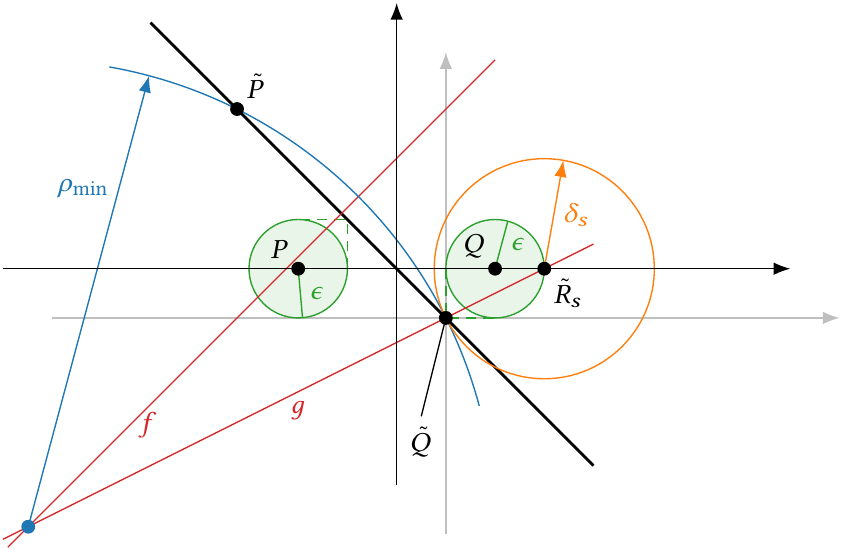}
  \caption{Worst-case scenario, where the ball with the enlarged
    {\color{Sec}maximal projection distance} around $\tilde R_s$ would
    touch the {\color{Pri}rotated ball} in $\tilde Q$.  Note that
    $\epsilon = \nicefrac h4$, $s=\epsilon$, and $\dist(\tilde P,
    \tilde Q) = h + 2 \epsilon$.  To compute the corresponding minimal
    radius $\rho_{\min}$ the perpendicular bisector $f$ of $\tilde P$
    and $\tilde Q$ and the line $g$ through $\tilde Q$ and $\tilde
    R_s$ is used.}
  \label{fig:min-step:cond}
\end{figure}

\begin{lemma}
  \label{lem:cond-rot}
  If $\sqrt {10} \, (h + 2 \epsilon) / 2 \le \rho$, $s \ge \epsilon$, and
  $\epsilon \le \nicefrac h4$ with $h \coloneqq \dist(P,Q)$, then the
  rotation~{\upshape\ref{it:fin-rot}} enlarges the maximal projection
  distance with respect to $\tilde R_s$.
\end{lemma}

\begin{proof}
  Figuratively, the maximal projection distance is enlarges as long as
  the circle of the maximal possible projection distance around
  $\tilde R_s$ touches the rotated $\rho$-ball on the right-hand side
  of $\tilde Q$.  More precisely, we consider the worst-case scenario
  where the circle of the maximal projection distance touches the
  $\rho$-ball in $\tilde Q$ for $s = \epsilon$ and
  $\epsilon = \nicefrac h 4$, see Figure~\ref{fig:min-step:cond}.
  After the rotation, the centre of the $\rho$-ball has to lie at the
  intersection of the lines $f$ and $g$.  Taking $\tilde Q$ as origin,
  we may describe the lines by the functions
  \begin{equation*}
    f(t) \coloneqq t + \tfrac{\sqrt 2}2 \, ( h + 2 \epsilon)
    \qquad\text{and}\qquad
    g(t) \coloneqq \tfrac t 2.
  \end{equation*}
  A brief computation now yields
  \begin{equation*}
    \rho
    = \pNormb{\bigl( \sqrt2\,( h + 2 \epsilon),
      \sqrt2 \,( h + 2 \epsilon) / 2 \bigr)^\T}
    = \tfrac{\sqrt 10}2 \, (h + 2 \epsilon).
  \end{equation*}
  For greater $\rho$, $s$ and smaller $\epsilon$, the disc of possible
  projection of $R_s$ touches the $\rho$-ball always on right-hand
  side of $\tilde Q$ during the rotation; so the maximal projection
  radius is enlarged.  \qed
\end{proof}

\begin{figure}[tp]
  \centering
  \includegraphics{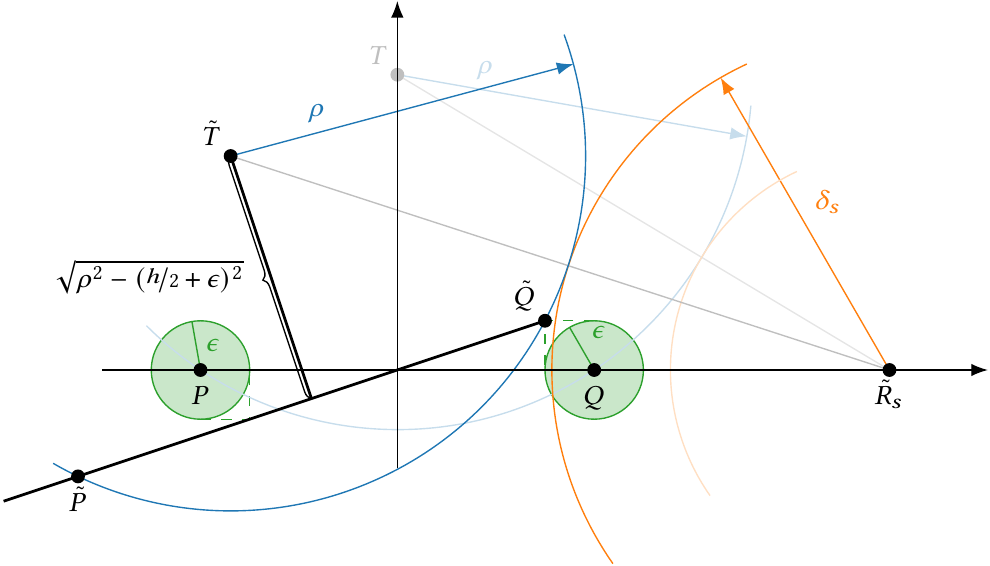}
  \caption{General situation after rotating a
    {\color{Pri}$\rho$-ball}, where the centre $\tilde T$ is located
    at the perpendicular bisector of $\tilde P$ and $\tilde Q$.  For
    comparison with Figure~\ref{fig:min-step}, which is indicated, we
    consider the mirrored ball in Figure~\ref{fig:min-step:rot}.  In
    order to estimate $\dist(Q, \proj_\gamma(\tilde R_s))$ from below,
    the radius $\rho$ has to be much greater than in the sketch.
  }
  \label{fig:min-step:err}
\end{figure}

After rotation a single $\rho$-ball, the scenery becomes like in
Figure~\ref{fig:min-step:err}.  Considering the union of all rotated
balls, which is rotationally symmetric around the line through $P$ and
$Q$, we see that every two-dimensional cut has this geometry.  Instead
of deriving an analysis similar to Figure~\ref{fig:min-step}, we
estimate how far $\tilde T$ in Figure~\ref{fig:min-step:err} is away
from $T$ in Figure~\ref{fig:min-step} and use the previous
results.

\begin{figure}[t]
  \centering
  \includegraphics{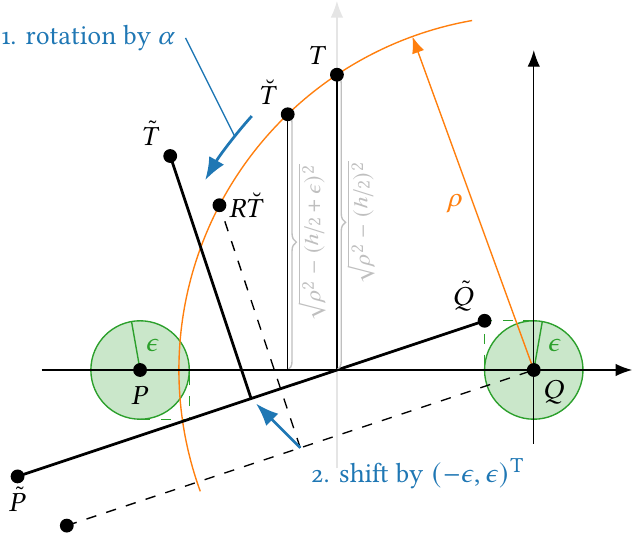}
  \caption{Calculating $\dist(\tilde T, T)$ by noticing that $\tilde T$
    results from $\breve T$ by rotation and translation.  The rotation
    angle $\alpha$ is given by
    $\tan \alpha = \epsilon / (\nicefrac h2 - \epsilon)$.}
  \label{fig:min-step:rad}
\end{figure}

\begin{lemma}
  \label{lem:dist-T}
  Let $T$ be as in Figure~\ref{fig:min-step} and $\tilde T$ as in
  Figure~\ref{fig:min-step:err}.  If
  $\sqrt{10} \, (h + 2\epsilon)/2 \le \rho$ and
  $\epsilon \le \nicefrac h 4$ with $h \coloneqq \dist(P,Q)$, then
  \begin{equation*}
    \dist(\tilde T , T)
    \le \bigl( \sqrt{32} \, \tfrac\rho h + 2\sqrt{2}
    \bigr) \, \epsilon.
  \end{equation*}
\end{lemma}

\begin{proof}
  Choosing $Q$ as reference point in Figure~\ref{fig:min-step:err},
  $\tilde T$ is essentially a rotation followed by a shift, see
  Figure~\ref{fig:min-step:rad}.  More precisely, we have
  \begin{equation*}
    \tilde T = R \breve T +
    (\begin{smallmatrix}
      - \epsilon \\ \epsilon
    \end{smallmatrix})
    \qquad\text{with}\qquad
    R =
    \begin{pmatrix}
      \cos \alpha & - \sin \alpha \\
      \sin \alpha & \cos \alpha
    \end{pmatrix},
  \end{equation*}
  where $\tan \alpha = \epsilon / (\nicefrac h 2 - \epsilon)$.
  Preliminary, we establish the following two estimates:
  \begin{enumerate}
  \item the \PFrobenius\ norm distance between rotation and identity
    is given by
    \begin{equation*}
      \FNormn{R - I}^2
      = 2 \, \bigl( (1 - \cos \alpha)^2 + (\sin \alpha)^2 \bigr)
      = 4 \, ( 1 - \cos \alpha)
      = 8 \, (\sin \nicefrac \alpha 2)^2;
    \end{equation*}
    further \PMollweide's formula \cite[Eq~(\Jupiter)]{Mol08} together
    with $\epsilon \le \nicefrac h 4$ implies
    \begin{equation*}
      \FNormn{R-I}^2
      \le 8 \, \Bigl( \tfrac{\epsilon}{(\nicefrac h2 - \epsilon) +
        \sqrt{(\nicefrac h2 - \epsilon)^2 + \epsilon^2}} \Bigr)^2
      \le 32 \, \tfrac{\epsilon^2}{h^2};
    \end{equation*}
  \item $T$ and $\breve T$ lie on the circle of radius $\rho$; since
    $\sqrt{10} \, (h + 2 \epsilon) / 2 \le \rho$ implies $\nicefrac h2 + \epsilon \le
    \nicefrac{\sqrt 2}2 \, \rho$, the vertical distance is less than the
    horizontal distance; we thus infer
    \begin{equation*}
      \pNormn{\breve T - T}
      \le \sqrt 2 \, \epsilon.
    \end{equation*}
  \end{enumerate}
  Exploiting these two bounds, we finally obtain
  \begin{align*}
    \pNormn{\tilde T - T}
    &\le \pNormn{R \breve T - \breve T}
      + \pNormn{\breve T - T}
      + \pNormn{
      \begin{smallmatrix}
        -\epsilon \\ \epsilon
      \end{smallmatrix}
    }
    \\[\fskip]
    &\le \FNormn{R - I} \, \rho
      + \pNormn{\breve T - T}
      + \sqrt 2 \, \epsilon
    \\[\fskip]
    &\le \sqrt{32} \, \tfrac \rho h \, \epsilon + 2\sqrt{2} \,
      \epsilon. 
      \tag*{\qed}
  \end{align*}
\end{proof}

The essential statement behind \thref{lem:dist-T} is that the geometry
exploited to establish minimal and maximal step sizes does only
slightly change if the projection to obtain $P$ and $Q$ becomes
inaccurate.  Based on the projection error $\epsilon > 0$, we define
\begin{equation*}
  \delta_h \coloneqq \delta_h (\epsilon)
  = \bigl( \sqrt{32} \, \tfrac \rho h + 2 \sqrt{2}
    \bigr) \, \epsilon.
\end{equation*}

\begin{figure}[t]
  \centering
  \includegraphics{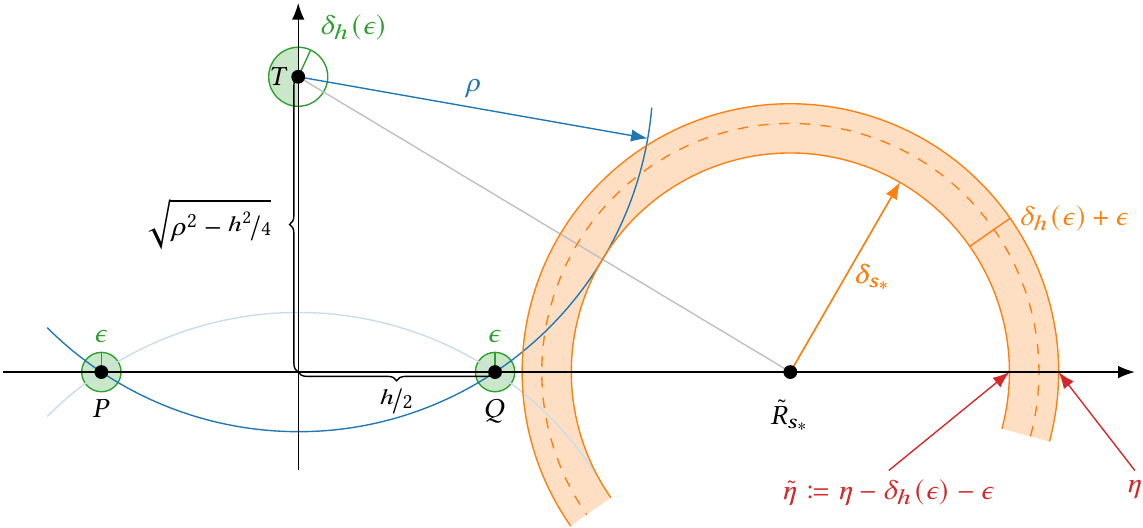}
  \caption{Incorporating the {\color{Ter}errors $\delta_h(\epsilon)$}
    and {\color{Ter}$\epsilon$} into the minimal and maximal step size
    derived in Theorem~\ref{thm:step-size-guar}.  Here $s_*$ is lessen to
    guarantee {\color{Qua}$\dist(Q,R) \le \tilde \eta$} without
    uncertainties.  Adding them, we have
    {\color{Qua}$\dist(Q,R) \le \eta$} at the most.  The minimal step
    size is adapted similarly.}
  \label{fig:min-step:final}
\end{figure}

\begin{theorem}[Step Size Guarantee]
  \label{thm:step-size-guar:inex}
  Let $\gamma \colon [0, \ell(\gamma)] \to \BR^d$ be a $\rho$-separated
  \PJordan\ $C^2$-arc, let $\epsilon > 0$ be the maximal numerical error of
  $\proj_\gamma$, and let $P$ and $Q$ be distinct points such that
  $\dist(P,\gamma) \le \epsilon$, $\dist(Q,\gamma) \le \epsilon$,
  $h \coloneqq \dist(P,Q) \le \eta$ for
  \raisebox{0pt}[0pt][0pt]{$\sqrt{10} \, (h + 2 \epsilon)/2 \le \rho$}, and
  $\epsilon \le \nicefrac h 4$.  If
  $s_* \coloneqq (\tilde\eta^2 + 2 \tilde\eta \rho) / (2 \rho
    + 2 \tilde\eta +h) \ge \epsilon$ where
  $\tilde \eta \coloneqq \eta - \delta_h(\epsilon) - \epsilon$, then
  the numerical projection
  $R \approx_\epsilon \proj_\gamma(Q + s_* v)$ with
  $v \coloneqq (PQ)^\rightarrow / \dist(P,Q)$
  satisfies
  \begin{equation*}
    \tfrac 6{80} \, \eta - \tfrac{86}{80} \, (\delta_h(\epsilon) + \epsilon)
    \le \dist(R,Q) \le \eta,
  \end{equation*}
  or, if the lower bound is positive and does not hold, $R$ is an
  approximate end point of $\gamma$.
\end{theorem}

\begin{proof}
  The assumptions of the theorem ensure that the necessary rotations
  leading to the geometry in Figure~\ref{fig:min-step:err} enlarge the
  maximal projection distance.  The difference to the situation in
  \thref{thm:step-size-guar} and Figure~\ref{fig:min-step} is that $T$
  is known only approximately resulting in an uncertainty of the
  maximal projection distance, see Figure~\ref{fig:min-step:final}.
  On the basis of the step size
  $s_* \coloneqq (\tilde\eta^2 + 2 \tilde\eta \rho) / (2 \rho + 2
  \tilde\eta +h)$ with
  $\tilde \eta \coloneqq \eta - \delta(\epsilon) - \epsilon$,
  \thref{thm:step-size-guar} ensures
  \begin{equation*}
    \tfrac{6}{80} \, \tilde \eta \le \dist(R,Q) \le \tilde \eta
  \end{equation*}
  for the error-free setting.  Including the uncertainty
  $\delta_h(\epsilon)$ of the maximal projection distance and the
  projection error $\epsilon$, we infer
  \begin{equation*}
    \tfrac{6}{80} \, \tilde \eta - \delta_h - \epsilon
    \le \dist(R,Q)
    \le \tilde \eta + \delta_h + \epsilon
  \end{equation*}
  yielding the assertion.  \qed
\end{proof}

The lower bound of the inexact step size guarantee depends on the
current $h$; so the lower bound may become arbitrary small by
iterating the result.  Proposing an additional assumption, we may
nevertheless guarantee that a slight modification of
\thref{alg:curve-app} recovers the underlying \PJordan\ curve up to a certain
\PHausdorff\ distance even if the projection to the curve cannot be
performed exactly.

\begin{theorem}[Termination]
  \label{thm:termi-inex}
  Let $\gamma \colon [0, \ell(\gamma)] \to \BR^d$ be a
  $\rho$-separated \PJordan\ $C^2$-arc, and let
  $\epsilon \in [0, \nicefrac \rho 4]$ be the maximal numerical error
  of $\proj_\gamma$.  Adapting \thref{alg:curve-app} by choosing
  \begin{equation*}
    \eta \coloneqq 2  \min \bigl\{ 10^{-\nicefrac 12} \rho
    , 2 \sqrt{\rho^2 - (\rho + \epsilon - E)^2}
    \bigr\} - 2 \epsilon
    \qquad\text{and}\qquad
    s \coloneqq s_h \coloneqq
    \frac{\tilde\eta^2_h + 2 \tilde\eta_h \rho}{2 \rho
    + 2 \tilde\eta_h +h}
  \end{equation*}
  with $E \in (\epsilon,\rho)$ and
  $\tilde \eta_h \coloneqq \eta -
  \delta_h(\epsilon) - \epsilon$ yields a polygonal chain
  $\tilde \gamma$ with $\dist(\tilde \gamma, \gamma) \le E$ in finite
  time, \ie\ \thref{alg:curve-app} terminates, if
  \;$\nicefrac{6}{80} \, \eta \le \dist(P_0, P_1) \le \eta$ and
  \begin{equation*}
    \tfrac 6{80}
    \, \eta -  \tfrac{86}{80} \, (2 \sqrt 2 +1) \, \epsilon
    \ge
    \bigl( \tfrac{86}{20} \, \sqrt{32} \, \rho \epsilon \bigr)^{\nicefrac12}.
  \end{equation*}
\end{theorem}

\begin{proof}
  The statement may be established analogously to \thref{thm:termi} as
  long as the guaranteed minimal step size
  \begin{equation*}
    \lambda(h) \coloneqq
    \tfrac 6{80} \, \eta - \tfrac{86}{80} \, (\delta_h(\epsilon) +
    \epsilon)
    = \tfrac{6}{80} \, \tilde \eta_h - \delta_h(\epsilon) - \epsilon
  \end{equation*}
  in \thref{thm:step-size-guar:inex} does not vanish or becomes
  negative.  First notice that $\delta_h(\epsilon)$ is inversely
  depending on $h$ such that $\lambda(h)$ and $\tilde \eta_h$
  become monotonically increasing with respect to $h$.  Let us now
  assume that there exists an $\tilde h$ such that
  \begin{equation*}
    4\epsilon \le \tilde h
    \le \lambda(\tilde h)
    \le \tfrac{6}{80} \, \tilde \eta_{\tilde h}.
  \end{equation*}
  For any $h \in [\tilde h, \eta]$, the step size $s_h$ is bounded
  from below by
  \begin{equation*}
    s_h
    \ge \tfrac25 \, \tilde \eta_h
    \ge \tfrac{6}{80} \,  \tilde \eta_{\tilde h}
    \ge 4 \epsilon
    \ge  \epsilon.
  \end{equation*}
  We further have $\epsilon \le \nicefrac{\tilde h}4 \le \nicefrac h4$
  and $\sqrt{10} \, (h + 2\epsilon) / 2 \le \rho$ due to the choice of
  $\eta$; so for each $h \in [\tilde h, \eta]$ the requirements of
  \thref{thm:step-size-guar:inex} are satisfied.  Since
  $h_1 \coloneqq \dist(P_0, P_1)$ is contained in $[\tilde h, \eta]$,
  \thref{thm:step-size-guar:inex} yields
  \begin{equation*}
    \eta \ge \dist(P_1,P_2)
    \ge \lambda(h_1)
    \ge \lambda(\tilde h)
    \ge \tilde h
  \end{equation*}
  and, iteratively, $\dist(P_k, P_{k+1}) \in [\tilde h, \eta]$.  For
  $\tilde h > 0$, we thus obtain the assertion.

  It remains to ensure the existence of $\tilde h > 0$.  A rearrangement
  of the inequality
  \begin{equation*}
    \tilde h
    \le \lambda(\tilde h)
    = \tfrac{6}{86} \, \eta - \tfrac{86}{80} \, \bigl( \sqrt{32} \,
    \tfrac \rho{\tilde h} + 2 \sqrt2 + 1 \bigr) \epsilon.
  \end{equation*}
  yields the quadratic inequality
  \begin{equation*}
    \tilde h^2 + \bigl( \tfrac{86}{80} \, (2\sqrt2 + 1) \, \epsilon -
    \tfrac6{80} \, \eta \bigr) \, \tilde h + \tfrac{86}{80} \,
    \sqrt{32} \, \rho\epsilon \le 0.
  \end{equation*}
  Due to the assumptions of the theorem, we know that the discriminant
  \begin{equation*}
    \Delta \coloneqq
    \bigl( \tfrac{86}{80} \, (2 \sqrt2 +1) \, \epsilon - \tfrac{6}{80}
    \, \eta \bigr)^2 - \tfrac{86}{20} \, \sqrt{32} \, \rho \epsilon
  \end{equation*}
  is non-negative, and that the
  centre of the non-empty solution interval satisfies
  \begin{equation*}
    \tilde h \coloneqq \tfrac 12 \, \bigl( \tfrac 6{80}
    \, \eta -  \tfrac{86}{80} \, (2 \sqrt 2 +1) \, \epsilon \bigr) 
    \ge \bigl(\tfrac{86}{20} \, \sqrt{32}
    \bigr)^{\nicefrac12} \, \epsilon
    \ge  4 \epsilon,
  \end{equation*}
  where we exploit $\rho \ge 4 \epsilon$ and again the last
  assumption.  The existence of an appropriate $\tilde h > 0$ is thus
  guaranteed.  \qed
\end{proof}

\begin{remark}
  Due to the assumptions, the maximal step size $\eta$ is always
  positive.  If $\epsilon$ becomes small, then the additional
  assumption in \thref{thm:step-size-guar:inex} are always
  satisfiable.  The other way round, for a given $E$ and $\rho$, the
  theorem enable us to calculate an upper bound for the maximal
  projection error $\epsilon$ that guarantees the succuss of
  \thref{alg:curve-app} by solving a simple quadratic equation.  \qed
\end{remark}

\section{Identification of Sleeve Functions}
\label{sec:ident-curve-sleeve}

Besides the underlying \PJordan\ curve, the structure function $g$ has
to be approximated too.  Having a non-curve point $x$ and its
projection $y \coloneqq \proj_\gamma(x)$, we have immediate access to
$g$ via
\begin{equation*}
  g(t^2)
  = g \Bigl(\dist \bigl(y + t \, \tfrac{x - y}{\pNormn{x-y}}, \gamma
  \bigr)^2 \Bigr)
  = f\bigl(y + t \, \tfrac{x - y}{\pNormn{x-y}} \bigr)
\end{equation*}
for $t \ge 0$ until an ambiguity point is hit.  Since $\gamma$ is a
finite-length \PJordan\ arc, we henceforth assume
$\ran \gamma \subset B_{\nicefrac12}$ for simplicity.  Other domains
may be considered analogously.  Restricting our interest to an
approximation of $f$ on $B_{\nicefrac12}$, we only have to determine
$g$ on the interval $[0,1]$.  We have here two possibilities: either
$g$ is approximated directly or $t \mapsto g_2(t) \coloneqq g(t^2)$ is
approximated.  We choose the second approach because $g_2$ immediately
represent the slopes of the sleeve function.

For simplicity, we approximate $g_2$ by a linear spline with
equispaced knots.  Beneficially, this simplifies the inversion in
\thref{alg:proj-cur}, where we can take $g^{-1}_2(z)$ instead of the
step size $\sqrt t$.  The approximation of $g_2$ with step size
$\sigma > 0$ may be incorporated into \thref{alg:curve-app} for a
$\rho$-separated \PJordan\ curve in the following manner:
\begin{enumerate*}
\item From the start point $x_0$, the structure function $g_2$ is
  sampled equispaced in direction $\nabla f(x_0)$ until the curve or an
  ambiguity point is hit.  
\item This especially gives an approximation of $g_2$ on
  $[0,\rho]$ such that \thref{alg:curve-app} can be performed.
\item Determine the point $x = \gamma(t)$ with the largest norm.
  Note that $x \perp \dot \gamma(t)$, that $\proj_\gamma((1+s) x) = x$ for
  $s \ge 0$, and that the ray $\{(1+s) x : s \ge 0 \}$ cannot contain
  any ambiguity point; so the approximation of $g_2$ may be extended
  onto $[0,1]$ by sampling in this direction. 
\end{enumerate*}

\begin{algorithm}[Sleeve Function Approximation]
  \label{alg:sleeve-app}
  {\scshape Input:} $\rho, \sigma > 0$,
  $\epsilon \in [0, \nicefrac \rho 4]$, $E \in (\epsilon, \rho)$,
  $x_0 \in \BR^d$ where $\nabla f(x_0)$ is well defined.
  \\
  {\scshape Initial Approximation of $g_2$:}
  \begin{enumerate}[(1),nosep]
  \item Compute $v \coloneqq \nabla f(x_0) / \pNormn{\nabla f(x_0)}$.
  \item Collect $f_k \coloneqq f(x_0 - k \sigma v)$ for
    $k = 0,1,\dots$ until $f_k - f_{k-1} > 0$.
  \item Collect $f_{-\ell} \coloneqq f(x_0 + \ell \sigma v)$ for
    $\ell = 1,2,\dots$ until $\sigma \, (k + \ell - 1) \ge \rho$.
  \item Use $(f_j)_{j=-\ell+1}^{k-1}$ to approximate $g_2$ on
    $[0, \sigma \, (k+\ell-1)]$ by a linear spline $\tilde g_2$.
  \end{enumerate}
  {\scshape Underlying \PJordan\ Curve:}
  \begin{enumerate}[resume*]
  \item Apply \thref{alg:curve-app} modified as in
    \thref{thm:termi-inex} to obtain
    $\tilde \gamma \coloneqq \bigcup_{j=-m+1}^{n-2} P_{j} P_{j+1}$.
  \end{enumerate}
  {\scshape Final Approximation of $g_2$:}
  \begin{enumerate}[resume*]
  \item Determine $x \coloneqq \argmax_{\{P_j : j = -m+1, \dots, n-1\}}
    \pNormn{P_j}$.
  \item Improve $x$ by setting $y \coloneqq (1+
    \rho / \pNormn{x}) \, x$
    and $x \coloneqq \proj_\gamma(y)$.
  \item Compute $v \coloneqq (y - x) / \pNormn{y - x}$.
  \item Collect $f_{r} \coloneqq f(x + r \sigma v)$ for
    $r = k + \ell, k + \ell + 1,\dots$ until $\sigma r \ge 1$.
  \item Use these to extend the linear spline $\tilde g_2$ to
    approximate $g_2$ on $[0,1]$. 
  \end{enumerate}
  {\scshape Output:} Polygonal chain $\tilde \gamma$, linear spline
  $\tilde g_2$ approximating $f$.
\end{algorithm}

\begin{remark}
  Note that the collected samples gives only rough informations about
  the root of $f(x_0 - t v)$, which is located somewhere around
  $\sigma \, (k-1)$.  To improve the approximation of $g_2$ in (4),
  which is crucial to compute $\proj_\gamma$ by \thref{alg:proj-cur},
  we therefore apply the \PNewton--\PRaphson\ method.  During the
  numerical experiments, we usually require only a few iteration
  to find the root with sufficient accuracy.  \qed
\end{remark}

Due to the approximation of $g_2$, the projection computed by
\thref{alg:proj-cur} becomes inexact since the true step size
$g_2^{-1}(z)$ is not available.  Depending on the local derivative of
$g_2$ the approximation error is here amplified or reduced; therefore,
we again assume that the projection error is again bounded by an
appropriate $\epsilon >0$.  If the first and second derivative of the
true structure function $g$ are bounded, and if the step size $\sigma$
is chosen appropriately, then the total approximation error for the
sleeve function $f$ may be controlled.

\begin{theorem}[Approximation Error]
  \label{thm:tot-app-err}
  Let
  $\gamma \colon [0, \ell(\gamma)] \to B_{\nicefrac12} \subset \BR^d$
  be a $\rho$-separated \PJordan\ $C^2$-arc, and let $g \in C^2([0,1], \BR_+)$
  be strictly monotonically increasing with
  \begin{equation*}
    M_1 \coloneqq \sup_{t \in [0,1]} 2 \, \absn{t \, \dot g(t^2)} < \infty
    \qquad\text{and}\qquad
    M_2 \coloneqq \sup_{t \in [0,1]} \absn{\dot g(t^2) + 2 t^2 \,
      \ddot g(t^2)} <
    \infty. 
  \end{equation*}
  If the projection error in \thref{alg:proj-cur} is bounded by
  $\epsilon > 0$, and if the requirements of \thref{thm:termi-inex}
  are satisfied, then \thref{alg:sleeve-app} with $E$ and $\sigma$
  approximates $f = g_2 \circ \dist(\cdot, \gamma)$ by
  $\tilde f = \tilde g_2 \circ \dist( \cdot, \tilde \gamma)$
  satisfying
  \begin{equation*}
    \sup_{x \in B_{\nicefrac12}} \absn{f(x) - \tilde f(x)}
    \le  M_1 E + \tfrac{M_2}{8} \, \sigma^2.
  \end{equation*}
\end{theorem}

\begin{proof}
  The first derivatives of $t \mapsto g_2(t) \coloneqq g(t^2)$ are
  just
  \begin{equation*}
    \dot g_2(t) = 2t \, \dot g(t^2)
    \qquad\text{and}\qquad
    \ddot g_2(t) = 2 \dot g(t^2) + 4t^2 \, \ddot g(t^2).
  \end{equation*}
  Exploiting \thref{thm:termi-inex}, the mean value theorem, and the
  well-established approximation error for linear splines
  $\pNormn{g_2 - \tilde g_2}_\infty \le \nicefrac{\sigma^2
    \pNormn{\ddot g_2}_\infty}{8}$, see for instance \cite[Ch~6,
  §5.1]{HH91}, we have
  \begin{align*}
    \absn{f(x) - \tilde f(x)}
    &= \absn{g_2(\dist(x,\gamma)) - \tilde g_2(\dist(x, \tilde
      \gamma))}
    \\[\fskip]
    &\le \absn{g_2(\dist(x, \gamma)) - g_2(\dist(x, \tilde \gamma))}
      + \absn{g_2(\dist(x, \tilde \gamma)) - \tilde g_2(\dist(x,
      \tilde \gamma))}
    \\[\fskip]
    &\le M_1 E + \tfrac{M_2}8 \, \sigma^2.
      \tag*{\qed}
  \end{align*}
\end{proof}

\section{Numerical Examples}
\label{sec:numerical-example}

Besides the theoretical guarantees for the approximation of
curve-based sleeve functions by polygonal chains and linear splines,
we next present several examples to show that the established concepts
may be carried out numerically.  All algorithms and experiments have
been implemented in Julia\footnote{The Julia Programming Language --
  Version 1.5.0 (\url{https://docs.julialang.org})}.

Oddly enough, the main obstacles is here the numerical evaluation of
the true sleeve function $f(x) \coloneqq g(\dist(x, \gamma)^2)$ and
its derivatives---even if the profile $g$ and the underlying \PJordan\
curve $\gamma$ are known analytically---since all computations require
the projection onto $\gamma$.  The projection may be computed by
minimizing the function
\begin{align*}
  F(t)
  &\coloneqq \pNormn{x - \gamma(t)}^2
  \\
  \dot F(t)
  &= -2 \, \iProdn{x - \gamma(t)}{\dot \gamma(t)}
  \\
  \ddot F(t)
  &= 2 \, \pNormn{\dot \gamma(t)}^2 - 2 \, \iProdn{x - \gamma(t)}{\ddot\gamma(t)}
\end{align*}
over the interval $[0, \ell(\gamma)]$.  For this purpose, we use the
well-known \PNewton--\PRaphson\ method, \ie\
$t_{k+1} \coloneqq t_k - \nicefrac{\dot F(t_k)}{\ddot F(t_k)}$, where $F$ is
sampled equispaced to find an appropriate starting value $t_0$, and
where the end points are considered separately.  In general, the
projection of a point onto a parametric curve is a non-trivial problem
by itself, see for instance \cite{SXSY14,HW05,LT95,LHLP+18} and
references therein.  In contrast, the projection to the polygonal
chain $\gamma$ is straightforward---project to each line segment and
take the global minimizer.

\paragraph{\PArchimedes[ean] Spiral}

\begin{figure}[tp]
  \centering
  \begin{tabular}{ll}
    \subfloat[Recovered curve $\tilde \gamma$.]
    {\includegraphics{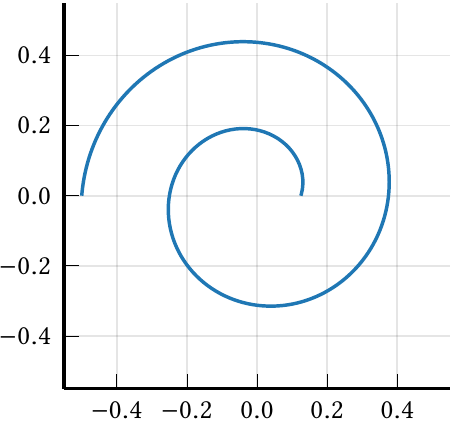}}
    & \subfloat[Recovered profile $\tilde g$.]
      {\hspace*{1.5pt}\includegraphics{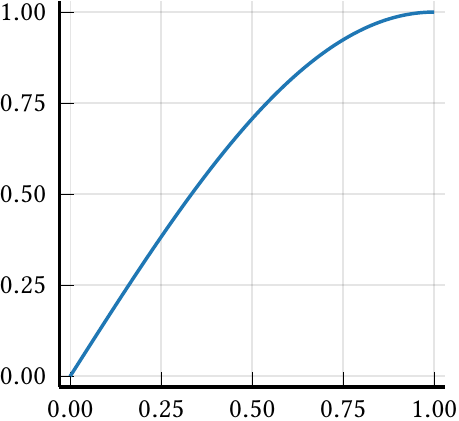}}
    \\
    \subfloat[Recovered sleeve function $\tilde f$.]
    {\includegraphics{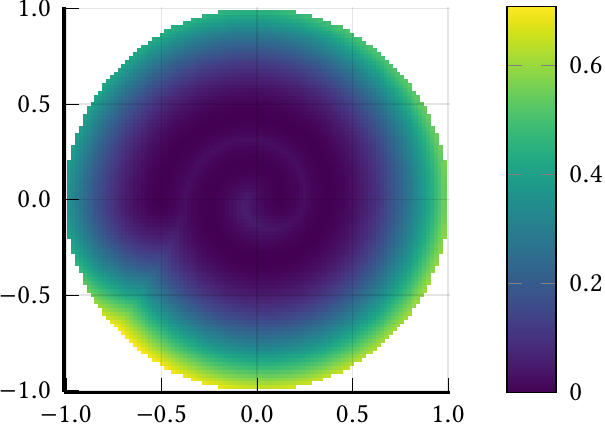}}
    & \subfloat[Absolute error $\absn{\tilde f - f}$.]
      {\includegraphics{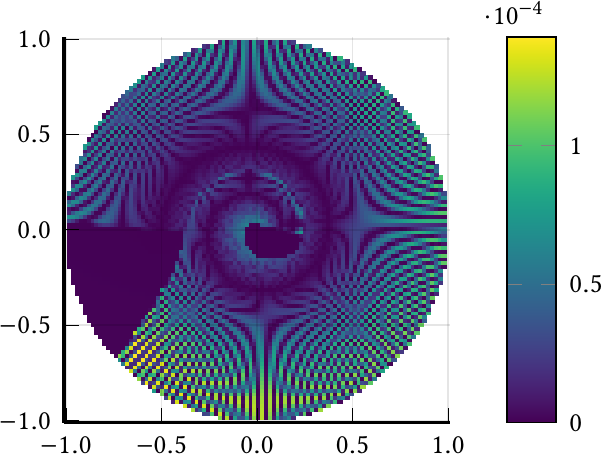}}
  \end{tabular}
  \caption{Sleeve function based on the \PArchimedes[ean] spiral and a
  sine profile.  The plots only show the recovered functions since
  they visually coincide with the true function.}
  \label{fig:spiral}
\end{figure}

The first considered sleeve function is based on the curve
\begin{equation*}
  \gamma \colon [0,1] \to \BR^2 :
  t \mapsto \bigl( \tfrac18 + \tfrac{3t}8 \bigr)
  \,
  \begin{pmatrix}
    \cos 3 \uppi t \\
    \sin 3 \uppi t
  \end{pmatrix},
\end{equation*}
which is a piece of an \PArchimedes[ean] spiral, and on the profile
function
\begin{equation*}
  g \colon [0,1] \to [0,1] :
  t \mapsto \sin(\nicefrac \uppi2 \, t).
\end{equation*}
The spiral is an $\nicefrac18$-separated $C^\infty$-curve.  Since the
maximal projection error is unknown, we perform \thref{alg:sleeve-app}
with $\epsilon \coloneqq 0$, \ie\ we assume that the error caused by
the approximation of the true profile is negligible.  For the
remaining parameters, we choose $E \coloneqq 10^{-3}$ and
$\sigma \coloneqq 10^{-4}$.  The result of the implemented methods is
shown in Figure~\ref{fig:spiral}.  Notice that the absolute
approximation error on $B_{\nicefrac12}$ is much smaller than $E$,
which bounds the error caused by the polygonal chain approximation of
the \PArchimedes[ean] spiral.  The error bound in
\thref{thm:tot-app-err} is here too pessimistic.

\paragraph{Three-dimensional Curve}

\begin{figure}[tp]
  \centering
  \subfloat[Recovered curve $\tilde \gamma$.]
  {\includegraphics{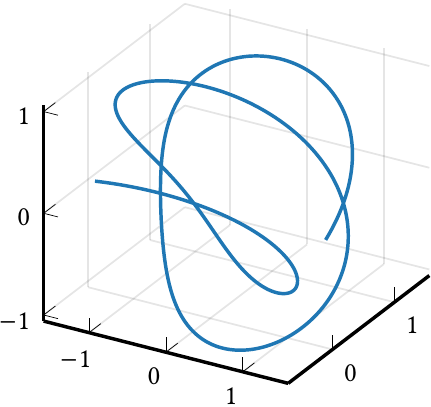}}
  \qquad
  \subfloat[Recovered profile $\tilde g$.]
  {\includegraphics{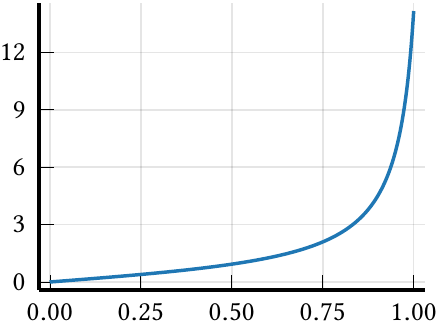}}
  \caption{Sleeve function based on a space curve and a tangent
    profile.  The plots only show the recovered functions since they
    visually coincide with the true function.}
  \label{fig:knot}
\end{figure}

Next, we consider the underlying curve
\begin{equation*}
  \gamma \colon [0,1] \to \BR^3 :
  t \mapsto
  \begin{pmatrix}
    (1 + \nicefrac12 \, \cos 4\uppi t) \, \cos 5\uppi t\\
    (1 + \nicefrac12 \, \cos 4\uppi t) \, \sin 5\uppi t\\
    \sin 4\uppi t
  \end{pmatrix}
\end{equation*}
together with the profile
\begin{equation*}
  g: [0,1] \mapsto \BR_+ :
  t \mapsto \tan (\nicefrac32 \, t).
\end{equation*}
The results of \thref{alg:sleeve-app} with $\epsilon \coloneqq 0$,
$E \coloneqq 10^{-2}$, and $\sigma \coloneqq 10^{-4}$ are shown in
Figure~\ref{fig:knot}.  Both---profile and curve---are well estimated.
Since the established theoretical approximation errors are completely
independent of the current dimension, we expect similar results for
the two one-dimensional approximation tasks in higher dimensions.

\paragraph{Finite differences}

\begin{figure}[t]
  \centering
    \subfloat[Recovered curve $\tilde \gamma$.]
    {\includegraphics{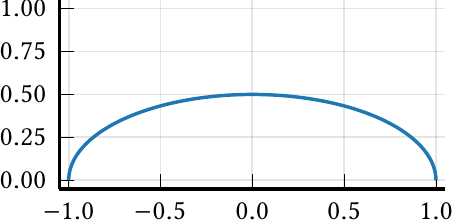}}
    \qquad
    \subfloat[Recovered profile $\tilde g$.]
      {\includegraphics{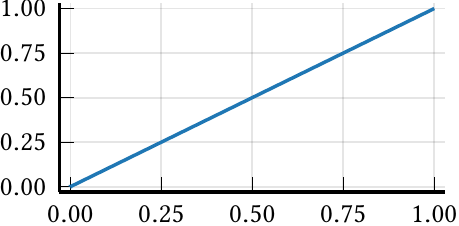}}
    \\
    \subfloat[Recovered sleeve function $\tilde f$ using symmetric differences.]
    {\includegraphics{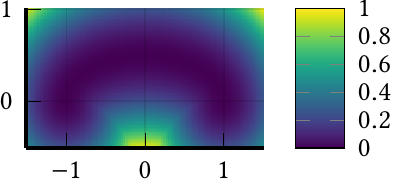}}
    \quad
     \subfloat[Absolute error using symmetric differences.]
     {\includegraphics{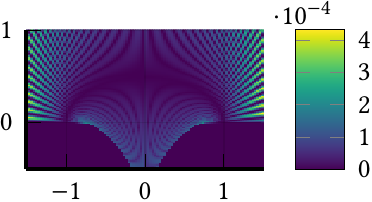}}
     \quad
     \subfloat[Absolute error using exact derivatives.]
      {\includegraphics{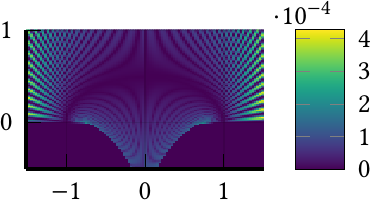}}
      \caption{Sleeve function based on a half-ellipse and an identity
        profile.  Note that the approximated profile $g_2$ is
        nevertheless non-linear.  The required derivatives are
        approximated by symmetric difference quotients.  The
        additional error is here marginal compared to the
        approximation using exact derivatives. }
  \label{fig:finite-diff}
\end{figure}

In the last two simulations, we consider the sleeve function with
respect to the $\nicefrac12$-separated half-ellipse
\begin{equation*}
  \gamma : [0,1] \to \BR^2 :
  t \mapsto
  \begin{pmatrix}
    \cos \uppi t\\
    \nicefrac12 \, \sin \uppi t
  \end{pmatrix}.
\end{equation*}
Up to now, we have assumed that we have access to the function values
and derivatives of the true sleeve function $f$.  If the derivatives
are not available, numerical differentiation may be used instead.  In
this numerical example, we approximate the required derivatives using
the symmetric difference quotient
\begin{equation*}
  [\nabla f(x)]_n \approx
  \frac{f(x + \tau e_n) - f(x - \tau e_n)}{2\tau}
\end{equation*}
for $n = 1, \dots, d$, where $e_n$ denotes the $n$th unit vector.
Letting the profile $g$ be the identity, which nevertheless results in
a locally non-linear sleeve function, and applying
\thref{alg:sleeve-app} with the parameter set $\epsilon \coloneqq 0$,
$E \coloneqq 10^{-3}$, $\sigma \coloneqq 10^{-4}$, and $\tau \coloneqq
10^{-8}$, we obtain the results in Figure~\ref{fig:finite-diff}.  The
additional error caused by the finite differences is here negligible.

\paragraph{Vanishing profile}

\begin{figure}[t]
  \centering
  \begin{tabular}{l@{\qquad}l}
    \subfloat[Comparison underlying curve.]
    {\hspace*{5pt}%
    \includegraphics{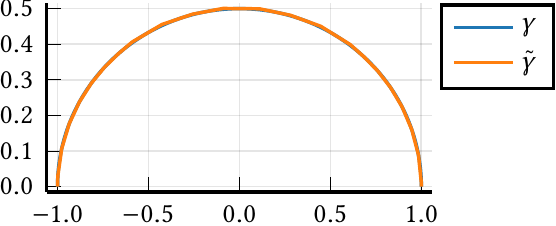}}
    & \subfloat[Comparison sleeve profile.]
      {\hspace*{1pt}%
      \includegraphics{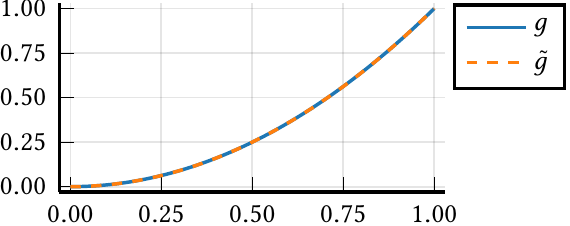}}
    \\
    \subfloat[Recovered sleeve function $\tilde f$.]
    {\includegraphics{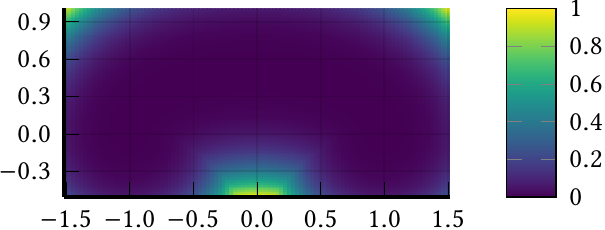}}
    & \subfloat[Absolute error $\absn{\tilde f - f}$]
      {\includegraphics{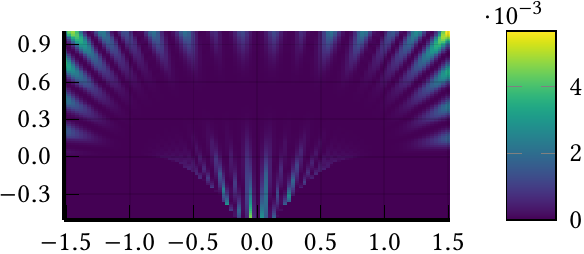}}
  \end{tabular}
  \caption{Sleeve function based on a half-ellipse and quadratic
    profile with vanishing derivative.  The loss of significance
    during the projection with Algorithm~\ref{alg:proj-cur} onto the
    unknown curve becomes visible---especially in the absolute error.}
  \label{fig:vanish-profile}
\end{figure}

In the last example, we consider again the half-ellipse but here
together with the profile function $g: t \mapsto t^2$.  Differently
form above, the derivative $\dot g$ is here not bounded away from zero on
$[0,\rho]$; so the required (again exact) derivative $\nabla f$ may
nearly vanish for small step sizes $s$.  The results of the
approximation with $\epsilon \coloneqq 0$, $E \coloneqq 10^{-2}$, and
$\sigma \coloneqq 10^{-4}$ is shown in
Figure~\ref{fig:vanish-profile}.  Notice that the approximation of the
underlying curve is much worse compared with the previous example.  If
$E$ becomes smaller, the polygonal chain starts to oscillate around
the true curve.  The projection error caused by the loss of
significance in \thref{alg:proj-cur} becomes clearly visible.  From a
numerical point, it is crucial that $\dot g$ is bounded from below away
from zero although this bound does not occur in the derived
approximation guarantees.

\section{Conclusion}
\label{sec:concl}

In order to capture a highly multivariate, curve-based sleeve
function, we propose an two-step method that reduces the recovery
problem to two one-dimensional approximation tasks.  The number of
samples required for the sleeve profile conforms to the
one-dimensional approximation theory.  The samples needed to find the
underlying curve mainly depend on its length and the well-separation.
The capturing of the underlying curve, which was the main interest of
the present work, progresses hand over hand along the curve.  Studying
the geometry, we guarantee a minimal and maximal step size, which
ensures that the algorithm cannot get stuck on the one hand and that
the wanted quality is achieved.  The derived error bounds for the two
proposed one-dimensional approximations tasks are completely
independence of the actual dimension.  In the moment, the algorithm is
based on linear approximations.  If the point queries are corrupted by
noise, these could be replaced by higher-order methods to improve the
performance.  The same applies to the required gradients, which have
been partly estimated by symmetric differences during the numerical
experiments.  For greater uncertainties in the function evaluations,
these could be replaced by higher-order differences or other
techniques like extrapolation may be used.  Finally, we exclusively
consider curve-based sleeve functions, so it is interesting to
generalize the proposed method to manifolds---either by approximating
the manifold by a suitable mesh or by some geometric multi-resolution
analysis \cite{ACM12}.  The presents results are a first step to study
general sleeve functions, and we believe that many
results---especially the minimal/maximal step sizes---can be
generalized to manifold-based sleeve functions.

\paragraph{Acknowledgment}

The author is especially grateful to Sandra \pers{Keiper}, the author
of \cite{Kei19}, for many fruitful discussions and for drawing my
attention to the topic of generalized ridge and sleeve functions.

%Moreover, the authors would like to thank the unknown referees for
%their valuable suggestions.

\printbibliography

\end{document}